\documentclass[12pt,leqno]{article}

\usepackage{amssymb,amsfonts,amsmath,amsthm,amscd,mathrsfs}
\usepackage{graphics,graphicx}
\def\IE{{\mathbb E}}
\def\IL{{\mathbb L}}
\def\IP{{\mathbb P}}
\def\IR{{\mathbb R}}

\def\IQ{{\mathbb Q}}
\def\IZ{{\mathbb Z}}

\def\chd{{\check{D}}}
\def\chv{{\check{V}}}
\def\chc{{\check{C}}}
\def\chu{{\check{U}}}

\def\n{\noindent}
\def\dsl{\textstyle\sum\limits}
\def\dil{\textstyle\int}

\def\dis{\displaystyle}
\def\o{\omega}
\def\fr{\mbox{\footnotesize $\dis\frac{1}{2}$}}

\def\ov{\overline}
\def\ve{\varepsilon}
\def\f{\footnotesize}
\def\r{\rightarrow}
\def\point{{\mbox{\large $.$}}}
\def\wh{\widehat}
\def\wt{\widetilde}

\def\cC{{\cal C}}

\def\cI{{\cal I}}

\def\cV{{\cal V}}

\dimendef\dimen=0

\newtheorem{theorem}{Theorem}[section]
\newtheorem{lemma}[theorem]{Lemma}

\newtheorem{corollary}[theorem]{Corollary}
\newtheorem{proposition}[theorem]{Proposition}
\newtheorem{remark}[theorem]{Remark}

\begin{document}

\baselineskip14pt
\noindent

\thispagestyle{empty}
\noindent

~ 
\vspace{2cm}
~

\bigskip
\begin{center}
{ \bf DISCONNECTION, RANDOM WALKS,  \\   AND RANDOM INTERLACEMENTS}
\end{center}



\begin{center}
Alain-Sol Sznitman
\end{center}


\vspace{0.5cm}
\begin{abstract}
We consider random interlacements on $\IZ^d$, $d\ge 3$, when their vacant set is in a strongly percolative regime. We derive an asymptotic upper bound on the probability that the random interlacements disconnect a box of large side-length from the boundary of a larger homothetic box. As a corollary, we obtain an asymptotic upper bound on a similar quantity, where the random interlacements are replaced by the simple random walk. It is plausible, but open at the moment, that these asymptotic upper bounds match the asymptotic lower bounds obtained by Xinyi Li and the author in \cite{LiSzni14}, for random interlacements, and by Xinyi Li in \cite{Li}, for the simple random walk. In any case, our bounds capture the principal  exponential rate of decay of these probabilities, in any dimension $d \ge 3$.
\end{abstract}

\vspace{4cm}
~

\n
Departement Mathematik\\ 
ETH-Z\"urich\\
CH-8092 Z\"urich\\
Switzerland

\newpage
\thispagestyle{empty}

~

\newpage
\setcounter{page}{1}

\setcounter{section}{-1}
\section{Introduction} 

How costly is it for simple random walk in $\IZ^d$, $d \ge 3$, to disconnect a box of large side-length $N$ from the boundary of a larger homothetic box? In this article, we obtain an asymptotic upper bound on this probability, as the application of a main result that provides an asymptotic upper bound on a similar quantity, where the random walk is replaced by random interlacements at a level $u$ such that the corresponding vacant set is in a strongly percolative regime. Although open at present, it is plausible that these asymptotic upper bounds are sharp, and respectively match the asymptotic lower bound of \cite{LiSzni14}, for random interlacements, and of \cite{Li}, for simple random walk. In any case, thanks to the current stand of knowledge concerning the strong percolative regime for the vacant set of random interlacements, see \cite{DrewRathSapo14a} (and also \cite{Teix11}, when $d \ge 5$), the bounds that we obtain in this article establish an exponential decay at rate $N^{d-2}$ of the above mentioned probabilities, in any dimension $d \ge 3$. This improves on \cite{Szni}, where such an exponential decay could only be ascertained in high enough dimension. The strategy that we employ here is nevertheless strongly influenced by the approach developed in \cite{Szni}. Yet, by several aspects, \cite{Szni} heavily relied on the specific nature of the Gaussian free field, and on the use of Gaussian bounds. One incidental interest of the present work is to uncover which objects in the present set-up correspond to the concepts introduced in the context of \cite{Szni}.

\medskip
We will now describe our results in a more precise form. We refer to Section 1 for further details concerning the various objects and notation. Given $u \ge 0$, we let $\cI^u$ stand for the random interlacements at level $u$ in $\IZ^d$, $d \ge 3$, and $\cV^u = \IZ^d \backslash \cI^u$, for the corresponding vacant set at level $u$. We denote by $\IP$ the probability governing the random interlacements. As $u$ increases, $\cV^u$ becomes thinner, and it is by now well-known (see \cite{Szni10}, \cite{SidoSzni09} or \cite{CernTeix12}, \cite{DrewRathSapo14c}), that there is a critical $u_* \in (0,\infty)$ such that
\begin{equation}\label{0.1}
\begin{array}{l}
\mbox{for $u < u_*$, $\IP$-a.s., $\cV^u$ has an infinite connected component},\\
\mbox{for $u > u_*$, $\IP$-a.s., all connected components of $\cV^u$ are finite}.
\end{array}
\end{equation}
Further, one can introduce a critical value
\begin{equation}\label{0.2}
u_{**} = \inf\{u \ge 0; \;\liminf\limits_L \;\IP[B_L \;\overset{^{\mbox{\footnotesize $\cV^u$}}}{\mbox{\Large $\longleftrightarrow$}} \hspace{-4ex}/ \quad \;\;\partial B_{2L}] = 0\},
\end{equation}

\n
where $B_L$ stands for the ball in the supremum norm with center $0$ and radius $L$ in $\IZ^d$, and $\partial B_{2L}$ for the boundary of $B_{2L}$ (see the beginning of Section 1), and the event under the probability refers to the existence of a nearest-neighbor path in $\cV^u$ starting in $B_L$ and ending in $\partial B_{2L}$. One knows (see \cite{PopoTeix13}, \cite{Szni12a}) that $u_{**}$ is finite and for $u > u_{**}$ the vacant set $\cV^u$ is in a strongly non-percolative regime, with a stretched exponential decay of the two-point function $\IP[0 \stackrel{\cV^u}{\longleftrightarrow} x]$ (in fact, an exponential decay, when $d \ge 4$, see \cite{PopoTeix13}). It is a simple fact that $u_* \le u_{**}$, but an important open problem whether the equality $u_* = u_{**}$ actually holds.

\medskip
In this article, we investigate the large $N$ asymptotic behavior of the probability of the disconnection event
\begin{equation}\label{0.3}
A_N = \{B_N \;\overset{^{\mbox{\footnotesize $\cV^u$}}}{\mbox{\Large $\longleftrightarrow$}} \hspace{-4ex}/ \quad \; \;S_N\},
\end{equation}

\n
where there is no nearest-neighbor path in $\cV^u$ from $B_N$ to $S_N = \{x \in \IZ^d; |x|_\infty = [MN]\}$, and where $|\cdot |_\infty$ stands for the sup-norm, $M > 1$ is a fixed (arbitrary) number, and $[MN]$ denotes the integer part of $MN$.

\medskip
In the strongly non-percolative regime $u > u_{**}$ for the vacant set, the event $A_N$ becomes typical for large $N$, and (see (\ref{1.44}))
\begin{equation}\label{0.4}
\lim\limits_N \;\IP[A_N] = 1, \;\mbox{when $u > u_{**}$}.
\end{equation}
On the other hand, when $u \le u_{**}$, the proof of Theorem 0.1 of \cite{LiSzni14} yields a lower bound:
\begin{equation}\label{0.5}
\liminf\limits_N \; \mbox{\f $\dis\frac{1}{N^{d-2}}$} \;\log \IP[A_N] \ge - \mbox{\f $\dis\frac{1}{d}$} \; (\sqrt{u}_{**} - \sqrt{u})^2 \,{\rm cap}_{\IR^d}([-1,1]^d),
\end{equation}

\n
where ${\rm cap}_{\IR^d}([-1,1]^d)$ stands for the Brownian capacity of $[-1,1]^d$ (see for instance \cite{PortSton78}, p.~58, for the definition).

\medskip
The main result of this article is contained in Theorem \ref{theo6.3} and provides an asymptotic upper bound on $\IP[A_N]$, in the strongly percolative regime $0 < u < \ov{u}$ of the vacant set, where $\ov{u}$ ($\le u_* \le u_{**}$) is a certain critical value, see (\ref{2.3}). Specifically, it is shown in Theorem \ref{theo6.3} that
\begin{equation}\label{0.6}
\limsup\limits_N \; \mbox{\f $\dis\frac{1}{N^{d-2}}$} \;\log \IP[A_N] \le - \mbox{\f $\dis\frac{1}{d}$} \; (\sqrt{\ov{u}}- \sqrt{u})^2 \,{\rm cap}_{\IR^d}([-1,1]^d), \;\mbox{for $0 < u < \ov{u}$}.
\end{equation}

\n
Crucially, one knows that $\ov{u} > 0$ in all dimensions $d \ge 3$, by the results of \cite{DrewRathSapo14a}. In addition, it is plausible, but open at the moment, that the inequalities $\ov{u} \le u_* \le u_{**}$ are equalities, i.e.~$\ov{u} = u_* = u_{**}$, so that one would actually infer from (\ref{0.5}), (\ref{0.6}) the asymptotic behavior
\begin{equation}\label{0.7}
\lim\limits_N \; \mbox{\f $\dis\frac{1}{N^{d-2}}$} \;\log \IP[A_N] = - \mbox{\f $\dis\frac{1}{d}$} \; (\sqrt{u_*} - \sqrt{u})^2 \,{\rm cap}_{\IR^d}([-1,1]^d), \;\mbox{for $0 < u <u_*$}.
\end{equation}

\n
In any case, (\ref{0.6}) improves on the results from Section 7 of \cite{Szni} that came as a consequence (via a Dynkin-type isomorphism) of upper bounds on similar disconnection probabilities by the sub-level-sets of the Gaussian free field. Indeed, for one thing, the asymptotic upper bounds of Section 7 of \cite{Szni}, due to the use of the isomorphism theorem, are not expected to match the lower bound (\ref{0.5}), moreover, in the present state of knowledge concerning the critical levels considered in \cite{Szni}, the results of \cite{Szni} only ensure an exponential decay at rate $N^{d-2}$ for $\IP[A_N]$, when the dimension is high enough.

\medskip
One can also compare (\ref{0.5}), (\ref{0.6}) to corresponding results for supercritical Bernoulli percolation. Unlike what happens in the present context, disconnecting $B_N$ from $S_N$ in the percolative phase would involve an exponential cost proportional to $N^{d-1}$ (and surface tension), in the spirit of the study of the existence of a large finite cluster at the origin, see p.~216 of \cite{Grim99}, and Theorem 2.5, p.~16 of \cite{Cerf00}.

\medskip
As an application of the main Theorem \ref{theo6.3} that proves (\ref{0.6}), we obtain, as a corollary, an asymptotic upper bound on the probability of a similar disconnection by simple random walk. If $\cV$ stands for the complement of the set of sites in $\IZ^d$ visited by the simple random walk, and $P_0$ for the probability governing the walk starting at the origin, there is a natural coupling of $\cV$ under $P_0$ and $\cV^u$ under $\IP[\cdot \,|\, 0 \in \cI^u]$ that ensures that $\cV^u \subseteq \cV$ (this coupling was already used in Section 7 of \cite{Szni}). Thus, letting successively $N$ tend to infinity and $u$ to $0$, we show in Corollary \ref{cor6.4} that
\begin{equation}\label{0.8}
\limsup\limits_N \; \mbox{\f $\dis\frac{1}{N^{d-2}}$} \;\log P_0 [B_N \;\overset{^{\mbox{\footnotesize $\cV$}}}{\mbox{\Large $\longleftrightarrow$}} \hspace{-4ex}/ \quad \;\;S_N] \le - \mbox{\f $\dis\frac{1}{d}$} \;\ov{u} \,{\rm cap}_{\IR^d}([-1,1]^d).
\end{equation}

\n
This yields a sharper bound on $\IP[A_N]$ than the application of (\ref{1.10}) and Theorem 6.3 of \cite{Wind08b}. It also improves on the results of Section 7 of \cite{Szni} (for the reasons explained below (\ref{0.7})). In addition, the article \cite{Li} establishes the lower bound
\begin{equation}\label{0.9}
\liminf\limits_N \; \mbox{\f $\dis\frac{1}{N^{d-2}}$} \;\log P_0 [B_N \;\overset{^{\mbox{\footnotesize $\cV$}}}{\mbox{\Large $\longleftrightarrow$}} \hspace{-4ex}/ \quad \;\;S_N] \ge - \mbox{\f $\dis\frac{1}{d}$} \;u_{**} \,{\rm cap_{\IR^d}}([-1,1]^d).  
\end{equation}

\n
In particular, if $\ov{u}$ and $u_{**}$ coincide (and $\ov{u} = u_* = u_{**})$, the combination of (\ref{0.8}) and (\ref{0.9}) would then show that
\begin{equation}\label{0.10}
\lim\limits_N \; \mbox{\f $\dis\frac{1}{N^{d-2}}$} \;\log P_0 [B_N \;\overset{^{\mbox{\footnotesize $\cV$}}}{\mbox{\Large $\longleftrightarrow$}} \hspace{-4ex}/ \quad \;\; S_N] = - \mbox{\f $\dis\frac{1}{d}$} \;u_{*} \,{\rm cap}_{\IR^d}([-1,1]^d).
\end{equation}

\n
As an aside, both asymptotic lower bounds (\ref{0.5}) and (\ref{0.9}) are proved by the application of the change of probability method. This involves certain probability measures $\wt{\IP}_N$ (in the case of (\ref{0.5})) and $\wt{P}_N$ (in the case of (\ref{0.9})) implementing suitable ``strategies'' to produce disconnection. If (and of course this point is open at the moment) the critical values $\ov{u} \le u_* \le u_{**}$ are identical, the results of the present article show that these strategies are (near) optimal, and thus hold special significance. With this in mind, let us say a word about these measures and the disconnection strategies they implement.

\medskip
In the case of random interlacements, i.e. (\ref{0.5}) and \cite{LiSzni14}, the measures $\wt{\IP}_N$ correspond to so-called {\it tilted interlacements} that can be viewed as slowly space-modulated random interlacements at a level (that slowly varies over space) equal to $f^2_N(x) = (\sqrt{u} + (\sqrt{u}_{**} - \sqrt{u}) \,h(\frac{x}{N}))^2$, $x \in \IZ^d$, where $h$ on $\IR^d$ is the solution of the equilibrium problem
\begin{equation}\label{0.11}
\left\{ \begin{array}{l}
\Delta h = 0 \;\mbox{outside $[-1,1]^d$}
\\[0.5ex]
\mbox{$h = 1$ on $[-1,1]^d$ and $h$ tends to $0$ at infinity}.
\end{array}\right.
\end{equation}

\n
Roughly speaking, the tilted interlacements create a ``fence'' around $B_N$, on which they locally behave as interlacements at level $u_{**}$ (actually one picks $u_{**} + \ve$ in place of $u_{**}$ in the formula for $f_N$). They induce locally on this fence a strong non-percolative regime for the vacant set, and thus typically disconnect $B_N$.

\medskip
Informally, the tilted interlacements correspond to a Poisson cloud of bilateral trajectories, where the motion of a particle is governed by the generator $\wt{L}_N g(x) =$ \linebreak $ \frac{1}{2d} \sum_{|x' - x| = 1} \;\frac{f_N(x')}{f_N(x)} \;(g(x') - g(x))$, instead of the discrete Laplacian (corresponding to the replacement of $f_N$ by a constant function in the above formula) in the case of the usual random interlacements. The tilted interlacements actually come up as a strategy to ensure at a preferential entropic cost an expected occupation time that varies over space and equals $f^2_N(x)$ at site $x$ in $\IZ^d$, instead of the constant value $u$ for the interlacement at level $u$. Remarkably, this time constraint induces tilted interlacements that have geometrical traces, which behave as random interlacements with a slowly space-modulated level $f^2_N(x)$. Incidentally, the tilted interlacements possibly offer in a discrete context a microscopic model for the type of ``Swiss cheese'' picture advocated in \cite{VanbBoltHoll01} for the moderate deviations of the volume of the Wiener sausage (however, the relevant modulating functions of \cite{VanbBoltHoll01} are different from those that appear in relation to (\ref{0.5})).

\medskip
In the case of simple random walk, i.e.~(\ref{0.9}) and \cite{Li}, the measures $\wt{P}_N$ correspond to {\it tilted walks} that, informally, behave as the walk started at the origin with generator $\ov{L}_N \,g(x) = \frac{1}{2d} \,\sum_{|x' - x| = 1}  \frac{h_N(x')}{h_N(x)}\, (g(x') - g(x))$ up to the deterministic time $T_N$, and afterwards evolve as simple random walk, where now $h_N(x) = h(\frac{x}{N})$, with $h$ as in (\ref{0.11}) and $T_N$ chosen such that by time $T_N$, the expected time spent by the tilted walk at a point $x \in B_N$ is $u_{**} \,h^2_N(x) = u_{**}$ (by the choice of $h$). Again, remarkably, this creates a ``fence'' around $B_N$, where the vacant set left by the tilted walk by time $T_N$ is locally in a strongly non-percolative regime (one actually uses $u_{**} + \ve$ in place of $u_{**}$, and a compactly supported approximation of $h$ in (\ref{0.11})). In this fashion, $\wt{P}_N$ ensures that with high probability $B_N$ gets disconnected from $S_N$ by the trace of the walk.

\bigskip
We will now present a rough outline of the proof of the main Theorem \ref{theo6.3} that establishes (\ref{0.6}). The general strategy is similar to \cite{Szni}. A quite substantial coarse graining takes place, and informally goes as follows. One considers  ``columns'' of boxes of side-length $L$ (of order $(N \log N)^{\frac{1}{d-1}}$) going from the surface $\{x \in \IZ^d; |x|_\infty = N\}$ of $B_N$ to the surface $S_N$ of $B_{MN}$ (for simplicity, assume $M = 2$). The number of such columns has roughly order $(\frac{N}{L})^{d-1} = \frac{N^{d-2}}{\log N}$. For each box $B$ sitting in the columns, one considers all successive excursions $Z^D_\ell, \ell \ge 1$, going from $D$ to $\partial U$, in the full collection of random interlacements of arbitrary levels, see (\ref{1.41}), where $D$ is a slightly larger box, concentric with $B$, and $U$ a much larger box, concentric with $B$, see (\ref{2.9}), (\ref{2.10}). One has ``good decoupling'' properties of the excursions $Z^D_\ell$, when the boxes are sufficiently far apart. The soft local time technique of \cite{PopoTeix13}, especially in the form developed in the Section 2 of \cite{ComeGallPopoVach13}, offers a very convenient tool to express these properties, see Proposition \ref{prop4.1}, (\ref{4.9}) and Remark \ref{rem4.3}. One shows in Theorem \ref{theo5.1} that in ``almost all columns'', for all boxes within the column, the corresponding excursions $Z^D_\ell$ that originate from random interlacements, with $\ell$ slightly below $\ov{u} \,{\rm cap}(D)$, where ${\rm cap}(D)$ stands for the random walk capacity of $D$, leave a vacant set in $D$ that ``percolates well'', and spend a substantial collective time in $D$, except on an event with super-exponentially decaying probability (at rate $N^{d-2}$). On the other hand, when disconnection by $\cI^u$ occurs (i.e.~$A_N$ is realized), each column must be blocked for the percolation within $\cV^u$. This forces the existence in most columns of a box where the number $N_u(D)$ of excursions from $D$ to $\partial U$ in the interlacement at level $u$ ``essentially'' exceeds $\ov{u} \,{\rm cap}(D)$. After a selection of such boxes, a step with not too high combinatorial complexity, thanks to our choice of $L$, we can use the occupation-time estimates of Section 3 (see Theorem \ref{theo3.2}) to bound $\IP[A_N]$ in essence by $\exp\{- (\sqrt{\ov{u}} - \sqrt{u})^2 \inf_\cC \;{\rm cap}(C) + o(N^{d-2})\}$, where $C = \bigcup_{B \in \cC} D$ and $\cC$ runs over the various collections of selected boxes $B$. Using a projection on the surface of $B_N$ and a Wiener-type criterion, as in \cite{Szni}, one obtains an asymptotic lower bound on ${\rm cap}(C)$ in terms of ${\rm cap}(B_N)$ uniformly over $\cC$, and (\ref{2.6}) quickly follows.

\medskip
We will now describe the organization of this article. Section 1 introduces further notation and recalls various facts concerning random walks, potential theory, and random interlacements. In Section 2 we introduce the strongly percolative regime $u < \ov{u}$ for the vacant set of random interlacements, and the notion of good boxes. We show in Theorem \ref{theo2.3} a super-polynomial decay in $L$ of the probability that a box $B$ is bad at levels below $\ov{u}$. In Section 3 we develop the occupation-time bounds that enter the proof of (\ref{0.6}). The main statement is contained in Theorem \ref{theo3.2}. Section 4 prepares the ground for the next section. We recall some facts about soft local times from  \cite{ComeGallPopoVach13} and set up couplings of the excursions within the interlacements, with independent collections of i.i.d. excursions. The main controls are contained in Proposition \ref{prop4.1}, see also Remark \ref{rem4.3}. In Section 5 we show the super-exponential decay at rate $N^{d-2}$ of the probability of existence of more than a few columns containing a bad box at levels below $\ov{u}$, see Theorem \ref{theo5.1}. In Section 6 we prove the main Theorem \ref{theo6.3} that establishes (\ref{0.6}), and derive (\ref{0.8}) in Corollary \ref{cor6.4}. The Appendix contains the proof of Lemma \ref{lem1.3} from Section 1.

\medskip
Finally, let us state the convention we use concerning constants. We denote by $c,c', \wt{c}$ positive constants changing from place to place that simply depend on $d$. Numbered constants such as $c_0,c_1,\dots$ refer to the value corresponding to their first appearance in the text. Dependence of constants on additional parameters appears in the notation.

\section{Notation and some useful facts} 
\setcounter{equation}{0}

In this section we introduce further notation and collect various facts concerning random walks, potential theory, and random interlacements. In particular, Propositions \ref{prop1.4} and \ref{prop1.5} contain estimates concerning equilibrium measures and entrance distributions that will be especially useful in Sections 3 and 4. Throughout, we tacitly assume that $d \ge 3$.

\medskip
We begin with some notation. For $s,t$ real numbers, we write $s \wedge t$ and $s \vee t$ for the minimum and the maximum of $s$ and $t$, and denote by $[s]$ the integer part of $s$, when $s$ is non-negative. We write $| \cdot |$ and $| \cdot |_\infty$ for the Euclidean and the $\ell^\infty$-norms on $\IR^d$. Given $x \in \IZ^d$ and $r \ge 0$, we let $B(x,r) = \{y \in \IZ^d; |y - x|_\infty \le r\}$ stand for the closed $\ell^\infty$-ball of radius $r$ around $x$. We say that a subset $B$ of $\IZ^d$ is a box when it is a translate of some set $\IZ^d \cap [0,L)^d$, with $L \ge 1$. We often write $[0,L)^d$ in place of $\IZ^d \cap [0,L)^d$, when no confusion arises. Given $A,A'$ subsets of $\IZ^d$, we denote by $d(A,A') = \inf\{|x - x'|_\infty$; $x \in A, x'\in A'\}$ the mutual $\ell^\infty$-distance between $A$ and $A'$. When $A = \{x\}$, we write $d(x,A')$ for simplicity. We let ${\rm diam}(A) = \sup\{|x - x'|_\infty$; $x,x' \in A\}$ stand for the $\ell^\infty$-diameter of $A$, and $|A|$ for the cardinality of $A$. We write $A \subset \subset \IZ^d$ to state that $A$ is a finite subset of $\IZ^d_\point$. We denote by $\partial A = \{y \in \IZ^d \backslash A$; $\exists x \in A$, $|y - x| = 1\}$, and $\partial_i A = \{x \in A$; $\exists y \in \IZ^d \backslash A$, $|y - x| = 1\}$, the boundary, and the internal boundary of $A$. For $f,g$ functions on $\IZ^d$, we write $f_+ = \max\{f,0\}$, $f_- = \max\{ - f,0\}$, and $\langle f,g\rangle = \sum_{x \in \IZ^d} f(x) g(x)$, when the sum is absolutely convergent. Incidentally, we also use the notation $\langle \rho, f\rangle$ for the integral of a function $f$ (on an arbitrary space) with respect to a measure $\rho$, when this quantity is meaningful.

\medskip
We continue with some notation concerning connectivity properties. We say that $x,y \in \IZ^d$ are neighbors, and sometimes write $x \sim y$, when $|y - x| = 1$. We call $\pi$: $\{0,\dots,n\} \rightarrow \IZ^d$ a path, when $\pi(i) \sim \pi(i-1)$, for $1 \le i \le n$. Given $A,B, U$ subsets of $\IZ^d$, we say that $A$ and $B$ are connected in $U$ and write $A \stackrel{U}{\longleftrightarrow} B$, when there exists a path with vales in $U ( \subseteq \IZ^d)$, which starts in $A$ and ends in $B$. When no such path exists, we say that $A$ and $B$ are not connected in $U$, and write $A \stackrel{U}{\nleftrightarrow} B$ (as in (\ref{0.3})).

\medskip
We now introduce some path spaces, and the set-up for continuous-time simple random walk. We consider $\wh{W}_+$ and $\wh{W}$ the spaces of infinite, respectively doubly infinite, $\IZ^d \times (0,\infty)$-valued sequences, such that the first coordinate of the sequence forms an infinite, respectively doubly infinite, nearest-neighbor path in $\IZ^d$, spending finite time in any finite subset of $\IZ^d$, and the sequence of second coordinates has an infinite sum, respectively infinite ``forward'' and ``backward'' sums. The second coordinate is meant to describe the duration of each step corresponding to the first coordinate. We write $\cal{\wh{W}}_+$ and $\cal{\wh{W}}$ for the respective $\sigma$-algebras generated by the coordinate maps. We denote by $P_x$ the law on $\wh{W}_+$ under which $Y_n$, $n \ge 0$, has the law of the simple random walk on $\IZ^d$, starting from $x$, and $\zeta_n$, $n \ge 0$, are i.i.d. exponential variables with parameter $1$, where $(Y_n, \zeta_n)_{n \ge 0}$ stand for the canonical $\IZ^d \times (0,\infty)$-valued coordinates on $\wh{W}_+$. We write $E_x$ for the corresponding expectation. Moreover, when $\rho$ is a measure on $\IZ^d$, we write $P_\rho$ and $E_\rho$, for the measure $\sum_{x \in \IZ^d} \rho(x) P_x$ (not necessarily a probability measure) and its corresponding ``expectation'' (that is, the integral with respect to the measure $P_\rho$).

\medskip
We attach to $w \in \wh{W}_+$ a continuous-time trajectory $X_t(w)$, $t \ge 0$, via
\begin{equation}\label{1.1}
\mbox{$X_t(w) = Y_k(w)$, for $t \ge 0$, when $\dsl^{k-1}_{i=0} \zeta_i(w) \le t < \dsl^k_{i=0} \zeta_i(w)$}
\end{equation}
(if $k = 0$, the sum on the left is understood as $0$).

\medskip
Thus, under $P_x$, the random trajectory $X_\point$ describes the continuous-time simple random walk with unit jump rate starting from $x$.

\medskip
Given $U \subseteq \IZ^d$ and $w \in \wh{W}_+$, we write $H_U(w) = \inf\{t \ge 0; X_t(w) \in U\}$, $T_U(w) = \inf\{t \ge 0; X_t(w) \notin U\}$ for the entrance time in $U$, and the exit time from $U$. Further, we let $\wt{H}_U(w) = \inf\{t \ge \zeta_1(w)$; $X_t(w) \in U\}$ stand for the hitting time of $U$.

\medskip
For $U \subseteq \IZ^d$, we write $\Gamma(U)$ for the space of right-continuous, piecewise-constant functions from $[0,\infty)$ to $U \cup \partial U$, with finitely many jumps on any finite time interval that  remain constant after their first visit to $\partial U$. The space $\Gamma(U)$ is endowed with the canonical shift of trajectories $(\theta_t)_{t \ge 0}$ and with the canonical coordinates, still denoted by $(X_t)_{t \ge 0}$. For $U \subset \subset \IZ^d$, the space $\Gamma(U)$ will be convenient to carry the law of certain excursions. We will also routinely view the law $P_x$ of the continuous-time simple random walk starting from $x$, as a measure on $\Gamma(\IZ^d)$. This will be convenient to carry out certain calculations.

\medskip
We now discuss some potential theory attached to the simple random walk. We write $g(\cdot,\cdot)$, for the simple random walk Green function, and $g_U(\cdot,\cdot)$, for the Green function of the walk killed upon leaving $U(\subseteq \IZ^d)$:
\begin{equation}\label{1.2}
g(x,y) = E_x\big[\dil^\infty_0 1\{X_s = y\}ds\big], \;g_U(x,y) = E_x \big[\dil^{T_U}_0 1\{X_s = y\}ds\big], \;x,y \in \IZ^d\,.
\end{equation}

\n
Both $g(\cdot,\cdot)$ and $g_U(\cdot,\cdot)$ are known to be finite and symmetric, and $g_U(\cdot,\cdot)$ vanishes if one of its arguments does not belong to $U$. When $f$ is a function on $\IZ^d$ such that $\sum_{y \in \IZ^d} g(x,y) | f(y)| < \infty$ for all $x$, in particular, when $f$ is finitely supported, we write
\begin{equation}\label{1.3}
Gf(x) = \dsl_{y \in \IZ^d} g(x,y) f(y), \; \mbox{for $x \in \IZ^d$}\,.
\end{equation}
Due to translation invariance, $g(x,y) = g(x-y,0) \stackrel{\rm def}{=} g(x-y)$, and one knows that (see Theorem 1.5.4, p.~31 of \cite{Lawl91})
\begin{equation}\label{1.4}
g(x) \sim c_0|x|^{2-d}, \; \mbox{as $|x| \rightarrow \infty$, where $c_0 = \mbox{\f $\dis\frac{d}{2}$} \; 
\Gamma\Big(\mbox{\f $\dis\frac{d}{2}$}-1\Big) \,\pi^{-\frac{d}{2}}$}.
\end{equation}

\medskip\n
As a direct consequence of the strong Markov property applied at the exit time of $U$, one has the following relation between the Green function and the killed Green function
\begin{equation}\label{1.5}
g(x,y) = g_U(x,y) + E_x [T_U < \infty, g(X_{T_U},y)], \;\mbox{for $x,y \in \IZ^d$}.
\end{equation}
Given $A \subset \subset \IZ^d$, we write $e_A$ for the equilibrium measure of $A$,
\begin{equation}\label{1.6}
e_A(x) = P_x [\wt{H}_A = \infty] \,1_A(x), \; \mbox{for $x \in \IZ^d$}
\end{equation}
(it is supported by the internal boundary of $A$), and ${\rm cap}(A)$ for the capacity of $A$, which is the total mass of $e_A$:
\begin{equation}\label{1.7}
{\rm cap}(A) = \dsl_{x \in A} e_A(x).
\end{equation}

\n
In the special case of the box $[0,L)^d$ one knows (see (2.16), p.~53 of \cite{Lawl91}) that
\begin{equation}\label{1.8}
c \,L^{d-2} \le {\rm cap}([0,L)^d) \le c' L^{d-2}, \;\mbox{for $L \ge 1$}.
\end{equation}
When $A \subset \subset \IZ^d$ is non-empty, we will denote by $\ov{e}_A$ the normalized equilibrium measure of $A$, namely
\begin{equation}\label{1.9}
\ov{e}_A(x) = \mbox{\f $\dis\frac{1}{{\rm cap}(A)}$} \, e_A(x), \; x \in \IZ^d .
\end{equation}

\n
We further recall (see Theorem T1, p.~300 of \cite{Spit01}) that for $A \subset \subset \IZ^d$
\begin{equation}\label{1.10}
P_x[H_A < \infty] = \dsl_{y \in A} g(x,y) \,e_A(y), \;\mbox{for $x \in \IZ^d$}.
\end{equation}

\n
Moreover, one has the sweeping identity (for instance as a consequence of (1.46) of \cite{Szni10}), stating that for $A \subset A' \subset \subset \IZ^d$,
\begin{equation}\label{1.11}
e_A(x) = P_{e_{A'}} [H_A < \infty, X_{H_A} = x], \; \mbox{for all $x \in \IZ^d$}.
\end{equation}

\n
We will also need to consider for $U \subseteq \IZ^d$, and $A$ finite subset of $U$, the equilibrium measure of $A$ relative to $U$ and the capacity of $A$ relative to $U$:
\begin{equation}\label{1.12}
e_{A,U}(x) = P_x [\wt{H}_A > T_U] \,1_A(x), x \in \IZ^d, \;\mbox{and} \; {\rm cap}_U(A) = \dsl_{x \in A} e_{A,U}(x).
\end{equation}

\n
We will now collect several results concerning continuous-time simple random walk that will be of use in the next sections. The first result shows a remarkable identification of the law of $\int^\infty_0 e_A(X_s)ds$ under $P_{\ov{e}_A}$ that will play a role in the proof of Theorem \ref{theo2.3} in the next section, when showing the super-polynomial decay of the probability that boxes are bad.

\begin{lemma}\label{lem1.1} Assume that $A \subset \subset \IZ^d$ is non-empty. Then, under 
$P_{\ov{e}_A}$
\begin{equation}\label{1.13}
\mbox{$\dil^\infty_0e_A (X_s) ds$ is distributed as an exponential variable of parameter $1$}.
\end{equation}
\end{lemma}

\begin{proof}
By Kac's moment formula (see Theorem 3.3.2, p.~74 of \cite{MarcRose06}), we know that for $n \ge 0$
\begin{equation}\label{1.14}
E_{\ov{e}_A} \big[\big(\dil^\infty_0 e_A(X_s)ds\big)^n\big] = n! \langle \ov{e}_A, (Ge_A)^n 1 \rangle,
\end{equation}

\medskip\n
where $(G e_A)^n$ stands for the $n$-th iteration of the linear operator $G e_A$, composition of $G$, see (\ref{1.3}), and multiplication by $e_A(\cdot)$, on the set of bounded functions on $\IZ^d$. By (\ref{1.10}) we see that $(G e_A) 1 = 1$ on $A$, so that $(G e_A)^n 1 = 1$ on $A$, for all $n \ge 0$, and hence
\begin{equation}\label{1.15}
E_{\ov{e}_A} \big[\big(\dil^\infty_0 e_A(X_s) ds\big)^n\big] = n! \;\mbox{for all $n \ge 0$}.
\end{equation}

\n
Thus, $\int^\infty_0 e_A(X_s)ds$ under $P_{\ov{e}_A}$ has the same moments as an exponential distribution of parameter $1$. By Theorem 3.9, p.~91 in Chapter 2 of \cite{Durr91}, this fact uniquely determines the law of $\int^\infty_0 e_A(X_s)ds$ under $P_{\ov{e}_A}$. The claim follows.
\end{proof}

The next two lemmas are, in essence, preparations for the main Propositions \ref{prop1.4} and \ref{prop1.5}, which provide controls on the equilibrium distribution and on the entrance distribution. These propositions will play an important role in Section 3, when deriving occupation-time estimates, see Proposition \ref{prop3.1}, and in Section 4, when setting up the coupling of excursions based on the soft-local times technique, see Proposition \ref{prop4.1}. The next lemma offers a comparison between the equilibrium measure and the relative equilibrium measure. We refer to the beginning of this section for the definition of a box in $\IZ^d$.

\begin{lemma}\label{lem1.2}
Let $U$ be a box in $\IZ^d$ and $A$ be a subset of $U$. Then,
\begin{equation}\label{1.16}
\mbox{$e_A$ and $e_{A,U}$ have the same support.}
\end{equation}
Moreover, one can uniquely define $\rho_{A,U}(\cdot)$ vanishing outside the support of $e_A$ (or $e_{A,U}$) so that
\begin{equation}\label{1.17}
e_{A,U}(y) = e_A(y)\big(1 + \rho_{A,U}(y)\big), \; \mbox{for} \; y \in \IZ^d.
\end{equation}
Further, one has
\begin{equation}\label{1.18}
\begin{array}{l}
1 \le 1 + \rho_{A,U}(y) \le \mbox{\f $\dis\frac{1}{p}$}, \; \mbox{for all $y \in \IZ^d$, where}
\\
p = \inf\limits_{\partial U} P_x[H_A = \infty] > 0 \vee \Big(1 - c \;\mbox{\f $\dis\frac{{\rm cap}(A)}{d(A, \IZ^d \backslash U)^{d-2}}$}\Big).
\end{array}
\end{equation}
\end{lemma}

\begin{proof}
By (\ref{1.6}), (\ref{1.12}) we know that $e_A \le e_{A,U}$. In addition, when $y$ belongs to the support of $e_{A,U}$, then
\begin{equation}\label{1.19}
e_{A,U}(y) \ge e_A(y) = P_y [\wt{H}_A = \infty] = E_y \big[\wt{H}_A > T_U, P_{X_{T_U}}[H_A = \infty]\big] \ge e_{A,U}(y) \,p.
\end{equation}

\n
Since $p > 0$ (recall $U$ is a box), we see that $y$ belongs to the support of $e_A$, and the claim (\ref{1.16}) follows. We can thus uniquely define $\rho_{A,U}$ (vanishing outside the common support of $e_A$ and $e_{A,U}$) so that (\ref{1.17}) holds. The first line of (\ref{1.18}) is an immediate consequence of (\ref{1.19}). As for the second line, we already observed that $p > 0$, and in addition, one has
\begin{equation}\label{1.20}
p = 1 - \sup\limits_{\partial U} P_x[H_A < \infty] > 1 - c \;\mbox{\f $\dis\frac{{\rm cap}(A)}{d(A, \partial U )^{d-2}}$}\,,
\end{equation}
by (\ref{1.10}), (\ref{1.7}), (\ref{1.4}). This completes the proof of (\ref{1.18}), and hence of Lemma \ref{lem1.2}.
\end{proof}

The next lemma states a decoupling effect that will play an important role in the proofs of Propositions \ref{prop1.4} and \ref{prop1.5} below. The proof is close to the arguments on p.~50 of \cite{Lawl91}, and can be found in the Appendix. We recall the end of the Introduction for the convention concerning constants.

\begin{lemma}\label{lem1.3}
There exist $c_1\ge 2$, $c_2 > 0$, such that for $K \ge c_1$, $L \ge 1$, and any $A \subseteq B(0,L)$, finite $U \supseteq B(0,KL)$, $y \in A, z \in \partial U$, one can find a unique $\psi^{A,U}_{y,z}$ with absolute value at most $c_2/K$ such that
\begin{equation}\label{1.21}
\begin{array}{l}
P_z[H_A < \wt{H}_{\partial U}, X_{H_A} = y] =  P_y[T_U < \wt{H}_A, X_{T_U} = z] = \\
 e_A(y) P_0[X_{T_U} = z] (1 + \psi^{A,U}_{y,z}),\; 
\mbox{and $\psi^{A,U}_{y,z} = 0$ if $e_A(y) \,P_y[X_{T_U} = z] = 0$}.
\end{array}
\end{equation}
\end{lemma}

The next proposition will be used in Section 3 in the proof of Proposition \ref{prop3.1}. It provides a bound on the difference of the equilibrium measure of a set $B$, when restricted to a well-separated piece $A$, with the equilibrium measure of the well-separated piece.

\begin{proposition}\label{prop1.4} (see Lemma \ref{lem1.3} for notation)

\medskip
If $L \ge 1$, $K \ge c_1$, then for any $A \subseteq B(0,L)$ and finite $B$ such that $B \cap B(0,KL) = A$, one has for all $y \in A$, setting $U = B(0,KL)$
\begin{equation}\label{1.22}
0 \le e_A(y) - e_B(y) = e_A(y) \,E_0[(1+ \psi^{A,U}_{y,X_{T_U}}), \;H_{B \backslash A} \circ \theta_{T_U} < \infty = H_A \circ \theta_{T_U}]
\end{equation}
and the term after $1$ in the parenthesis is a.s. smaller in absolute value than $c_2/K$.
\end{proposition}

\begin{proof}
For $y \in A$ we have 
\begin{equation}\label{1.23}
e_A(y) - e_B(y) = P_y[\wt{H}_A = \infty] - P_y[\wt{H}_B = \infty] = P_y[\wt{H}_A = \infty > H_{B \backslash A}]
\end{equation}
and the left-hand inequality of (\ref{1.22}) follows. Moreover, the last term above equals
\begin{equation}\label{1.24}
\begin{array}{l}
P_y[\wt{H}_A = \infty > H_{B \backslash A}] \stackrel{\rm strong \;Markov}{=} E_y\big[\wt{H}_A > T_U, P_{X_{T_U}}[H_{B \backslash A} < \infty = H_A]\big]
\\[1ex]
\dsl_{z \in \partial U} P_y[T_U < \wt{H}_A, X_{T_U} = z] \,P_z [H_{B \backslash A} < \infty = H_A] \stackrel{(\ref{1.21})}{=}
\\[1ex]
\dsl_{z \in \partial U} e_A(y) (1 + \psi^{A,U}_{y,z}) \,P_0[X_{T_U} = z] \,P_z[H_{B \backslash A} < \infty = H_A] \stackrel{\rm strong \;Markov}{=} 
\\[1ex]
e_A(y) \,E_0 [(1 + \psi^{A,U}_{y,X_{T_U}}), \; H_{B \backslash A} \circ \theta_{T_U} < \infty = H_A \circ \theta_{T_U}],
\end{array}
\end{equation}
where $|\psi^{A,U}_{y,X_{T_U}}| \le c_2/K$, $P_0$-a.s. by (\ref{1.21}). This proves the claim (\ref{1.22}).
\end{proof}

The next proposition will be important in the proof of Proposition \ref{prop3.1} and in Section 4, when setting up, by soft local time techniques, a coupling of excursions inside the interlacements, with a collection of independent excursions (see Proposition \ref{prop4.1}). It pertains to the entrance distribution of the walk starting from afar and conditioned to enter a set $B$ through a well-separated piece $A$.

\begin{proposition}\label{prop1.5} $(0 < \delta < 1)$

\medskip
If $L \ge 1$ and $K \ge c_3(\delta) \ge 2$, then for any non-empty $A \subseteq B(0,L)$ and finite $B$ such that $B \cap B(0,KL) = A$, and $\IZ^d \backslash B$ is connected, one has for $y \in A$ and $x \in \IZ^d \backslash (B \cup B(0,KL))$
\begin{align}
\big(1 - \mbox{\f $\dis\frac{\delta}{10}$}\big) \,\ov{e}_A(y) & \le P_x [X_{H_B} = y|H_B < \infty, X_{H_B} \in A] \le \big(1 + \mbox{\f $\dis\frac{\delta}{10}$}\big) \,\ov{e}_A(y), \label{1.25}
\\[2ex]
\big(1 - \mbox{\f $\dis\frac{\delta}{10}$}\big) \,\ov{e}_A(y) & \le  \mbox{\f $\dis\frac{e_B(y)}{e_B(A)}$} \le \big(1 + \mbox{\f $\dis\frac{\delta}{10}$}\big)\,\ov{e}_A(y) \quad \big(\mbox{with} \; e_B(A) = \dsl_{z \in A} e_B(z)\big). \label{1.26}
\end{align}
\end{proposition}

\begin{proof}
We first prove (\ref{1.25}) and start with a lemma.

\begin{lemma}\label{lem1.6}
For $A,B$ as above, and setting $U = B(0,KL)$, $\wt{U} = B^c$, for any $x \in \wt{U} \backslash U = \IZ^d \backslash (B \cup B(0,KL))$ and $y \in A$, one has
\begin{equation}\label{1.27}
P_x[H_B < \infty, X_{H_B} = y] = \dsl_{z \in \partial U} g_{\wt{U}} (x,z) \,P_z [H_A < \wt{H}_{\partial U}, X_{H_A} = y].
\end{equation}
\end{lemma}

\begin{proof}
The proof is similar to the proof of (2.3) on p.~48 of \cite{Lawl91}. The identity (\ref{1.27}) is obtained by considering the discrete skeleton of the walk and summing over the possible values of the time of the last visit to $\partial U$ before entering $B$ through $y$ in $A$, on the event in the left member of (\ref{1.27}).
\end{proof}

We now resume the proof of (\ref{1.25}) but keep the notation $U$ and $\wt{U}$ from Lemma \ref{lem1.6}. Since $\wt{U}$ is connected and $\partial A \subseteq \wt{U}$, it follows that for some $y$ in $A$ the left-hand side of (\ref{1.27}) is positive, and hence, $g_{\wt{U}}(x,\cdot)$ does not identically vanish on $\partial U$. As a result, the conditional probability in (\ref{1.25}) is well-defined, and for any $y \in A$ it equals
\begin{equation}\label{1.28}
\dis\frac{\dsl_{z \in \partial U} g_{\wt{U}}(x,z) \,P_z[H_A < \wt{H}_{\partial U}, X_{H_A} = y]}{\dsl_{z \in \partial U, y' \in A} g_{\wt{U}}(x,z) P_z[H_A < \wt{H}_{\partial U}, X_{H_A} = y']}\,.
\end{equation}

\n
When $K > c_1 \vee c_2$ (see Lemma \ref{lem1.3}), the probability in the numerator is at most $e_A(y)$ \,$P_0[X_{T_U} = z] (1 + c_2/K)$ and the sum over $y'$ of the probabilities in the denominator is at least ${\rm cap}(A) \,P_0[X_{T_U} = z] (1-c_2/K)$. Hence, the expression in (\ref{1.28}) is at most $\ov{e}_A(y) \,\frac{1 + c_2/K}{1 - c_2/K}$. In a similar fashion, we see that the expression in (\ref{1.28}) is at least $\frac{1 - c_2/K}{1 + c_2/K}$. The claim (\ref{1.25}) follows.

\medskip
We now turn to the proof of (\ref{1.26}). By Theorem 2.1.3, p.~51 of \cite{Lawl91}, see also (2.13), p.~52 of the same reference, one knows that when $x$ tends to infinity in (\ref{1.25}) $P_x[X_{H_B} = y|H_B < \infty]$ tends to $\ov{e}_B(y)$ for each $y$ in $B$. In particular, letting $x$ tend to infinity in (\ref{1.27}) (so that $x$ lies in $\wt{U} \backslash U = \IZ^d \backslash (B \cup B(0,KL))$ when $x$ is sufficiently large), we see that the conditional probability in (\ref{1.25}) tends to the ratio $\frac{\ov{e}_B(y)}{\ov{e}_B(A)} = \frac{e_B(y)}{e_B(A)}$, and the claim (\ref{1.26}) follows. This concludes the proof of Proposition \ref{prop1.5}.
\end{proof}

For the last result of this section concerning the continuous-time simple random walk, we consider $A \subseteq U$, with $U$ a box in $\IZ^d$, and the number of excursions of the walk from $A$ to the complement of $U$ (or as a shorthand, from $A$ to $\partial U$). Formally, we define
\begin{equation}\label{1.29}
R_1 = H_A \le D_1 \le R_2 \le \dots \le R_k \le D_k \le \dots \le \infty,
\end{equation}

\n
as the successive times of return to $A$ and departure from $U$ for $(X_t)_{t \ge 0}$, and note that $P_x$-a.s. on $\{R_k < \infty\}$, $D_k$ is finite, and each inequality is strict if the left member is finite $(x$ is an arbitrary point in $\IZ^d$). The number of excursions of the walk from $A$ to the complement of $U$ is defined as
\begin{equation}\label{1.30}
N_{A,U} = \sup\{k \ge 1; D_k < \infty\}.
\end{equation}

\begin{lemma}\label{lem1.7} ($A \subseteq U$, $U$ a box in $\IZ^d$)

\medskip
Set $p = \inf_{\partial U} P_y[H_A = \infty]$ as in (\ref{1.18}). Then, for $\lambda > 0$ such that $e^\lambda (1-p) < 1$, one has
\begin{equation}\label{1.31}
\sup\limits_{x \in A} E_x [e^{\lambda N_{A,U}}] \le \mbox{\f $\dis\frac{e^\lambda p}{1- e^\lambda (1-p)}$}\,.
\end{equation}
(we recall from (\ref{1.18}) that $p > 0 \vee (1 - c\; \frac{{\rm cap}(A)}{d(A,U^c)^{d-2}})$).
\end{lemma}

\begin{proof}
For simplicity we write $N$ in place of $N_{A,U}$ in this proof. Observe that for $x \in A$, $P_x$-a.s., $N = 1 + (N \circ \theta_{H_A} 1\{H_A < \infty\}) \circ \theta_{T_U}$.

\medskip
Hence, using the strong Markov property at time $T_U$ and then at time $H_A$, we obtain
\begin{equation}\label{1.32}
\begin{split}
E_x[e^{\lambda N}] & = E_x[e^\lambda (e^{\lambda N \circ \theta_{H_A} 1\{H_A < \infty\}}) \circ \theta_{T_U}]
\\[1ex]
& = e^\lambda \,E_x\big[E_{X_{T_U}}[e^{\lambda N \circ \theta_{H_A}} 1 \{H_A < \infty\} + 1\{H_A = \infty\}]\big]
\\[1ex]
& = e^\lambda \,E_x\big[E_{X_{T_U}}[ 1 \{H_A < \infty\} \,E_{X_{H_A}}[e^{\lambda N}] + 1\{H_A = \infty\}]\big].
\end{split}
\end{equation}
Thus, setting $\varphi = \sup_{x \in A} E_x [e^{\lambda N}] ( \ge 1)$, we find that
\begin{equation}\label{1.33}
\varphi \le e^\lambda \big((1- p) \varphi + p\big).
\end{equation}
The claim (\ref{1.31}) readily follows.
\end{proof}

\medskip
We will now recall some facts concerning continuous-time random interlacements. We refer to \cite{Szni12b} for more details. In the notation from the beginning of this section, we define $\wh{W}^* = \wh{W}/\sim$, where for $w,w' \in\wh{W}$, $w \sim w'$ means that $w(\cdot) = w(\cdot +k)$ for some $k \in \IZ$. We also consider the canonical map $\pi^*$: $\wh{W} \rightarrow \wh{W}^*$, and for $A \subset \subset \IZ^d$ denote by $\wh{W}_A^*$ the subset of $\wh{W}^*$ of trajectories modulo time-shift that intersect $A$. For $w^* \in \wh{W}_A^*$, we denote by $w^*_{A,+}$ the unique element of $\wh{W}^+$ that follows $w^*$ step by step from the first time it enters $A$.

\medskip
The continuous-time random interlacements can be constructed as a Poisson point process on the space $\wh{W}^* \times \IR_+$, with intensity measure $\wh{\nu}(dw^*)du$, where $\wh{\nu}$ is a $\sigma$-finite measure on $\wh{W}^*$ such that its restriction to $\wh{W}^*_A$ (denoted by $\wh{\nu}_A$) is equal to $\pi^* \circ \wh{Q}_A$, where $\wh{Q}_A$ is a finite measure on $\wh{W}$ such that letting $(X_t)_{t \in \IR}$ stand for the continuous-time process attached to $w \in \wh{W}$ (see (1.7) in \cite{Szni12b}), then
\begin{equation}\label{1.34}
\wh{Q}_A[X_0 = x] = e_A(x), \;\mbox{for $x \in \IZ^d$},
\end{equation}
and when $e_A(x) > 0$,
\begin{equation}\label{1.35}
\begin{array}{l}
\mbox{under $\wh{Q}_A$  conditioned on $X_0 = x$, $(X_t)_{t \ge 0}$ and the right-continuous}\\
\mbox{regularization of $(X_{-t})_{t > 0}$ are independent, and respectively distributed}\\
\mbox{as $(X_t)_{t \ge 0}$ under $P_x$, and $X$ after its first jump under $P_x[\cdot \,|\wt{H}_A = \infty]$}.
\end{array}
\end{equation}

\medskip\n
The space $\Omega$ on which the Poisson point measure is defined can be conveniently chosen as
\begin{equation}\label{1.36}
\begin{split}
\Omega = \{ & \mbox{$\omega = \sum_{i \ge 0} \delta_{(w^*_i, u_i)}$;  with $w^*_i \in \wh{W}^*$ for each $i \ge 0$, and $u_i > 0$ pairwise}
\\[-0.5ex]
&\mbox{distinct, and so that $\o(\wh{W}^*_A \times [0,u]) < \infty$ and $\o(\wh{W}^*_A \times \IR_+) = \infty$,}
\\[-0.2ex]
&\mbox{for any non-empty $A \subset \subset \IZ^d$ and $ u \ge 0\}$}.
\end{split}
\end{equation}

\n
The space $\Omega$ is endowed with the canonical $\sigma$-algebra and we denote by $\IP$ the law on $\Omega$ under which $\omega$ is a Poisson point process of intensivity measure $\wh{\nu} \otimes du$.

\medskip
Given $\o \in \Omega$, $A \subset \subset \IZ^d$ and $u \ge 0$, we define the point measure $\mu_{A,u}(\o)$ on $\wh{W}_+$ collecting for the $w_i^*$ with label $u_i$ at most $u$ that enter $A$ in the cloud $\o$, the onward trajectories after the first entrance in $A$:
\begin{equation}\label{1.37}
\mu_{A,u}(\omega)= \dsl_{i \ge 0} 1\{w^*_i \in \wh{W}^*_A, u_i \le u\} \,\delta_{(w^*_i)_{A,+}}, \;\mbox{if} \; \o = \dsl_{i \ge 0} \delta_{w_i^*,u_i}.
\end{equation}

\n
The key property of these point-measures, is that for any$ A \subset \subset \IZ^d$, $u \ge 0$,
\begin{equation}\label{1.38}
\mbox{under $\IP, \mu_{A,u}$ is a Poisson point process on $\wh{W}_+$ with intensity measure $u\,P_{e_{A}}$}.
\end{equation}

\medskip
Then, given $\o = \sum_{i \ge 0} \delta_{(w^*_i,u_i)}$ in $\Omega$ and $u \ge 0$, the random interlacement at level $u$, and the vacant set at level $u$, are now defined as the random subsets of $\IZ^d$
\begin{equation}\label{1.39}
\cI^u(\o) = \textstyle \bigcup\limits_{i : u_i \le u} \,{\rm range}(w_i^*), \; \; \cV^u(\o) = \IZ^d \backslash \cI^u(\o),
\end{equation}

\n
where for $w^* \in  \wh{W}^*$, range$(w^*)$ stands for the set of points in $\IZ^d$ visited by any $w \in \wh{W}$ with $\pi^*(w) = w^*$.

\medskip
Another object of interest for us is $L_{x,u}(\o)$, the occupation-time at site $x$ and level $u$ of random interlacements, that is, the total time spent at $x$ by all trajectories $w_i^*$ with label $u_i \le u$ in the cloud $\o = \sum_{i \ge 0} \delta_{(w^*_i,u_i)} \in \Omega$. So, for $V$: $\IZ^d \rightarrow \IR$ supported on $A \subset \subset \IZ^d$, one has $\langle  \mu_{A,u} ,  \dil^\infty_0 V(X_s) ds \rangle = \sum_{x \in \IZ^d} V(x) \,L_{x,u}$. Moreover, $\IE[L_{x,u}] = u$, and one has the following formula for the Laplace transform of $(L_{x,u})_{x \in \IZ^d}$, see Theorem 2.1 of \cite{Szni12d}. Namely, letting $\| \cdot \|_\infty$ stand for the supremum norm and with $G$ as in (\ref{1.3}),
\begin{equation}\label{1.40}
\begin{array}{l}
\mbox{for any finitely supported $V$: $\IZ^d \rightarrow \IR$ such that $\|GV\|_\infty < 1$, and $u \ge 0$}\\
\IE\big[\exp\big\{\dsl_{x \in \IZ^d} V(x) \,L_{x,u}\big\}\big] = \exp\{ u \langle V, (I - GV)^{-1} 1\rangle\}.
\end{array}
\end{equation}

\n
One actually knows more: there is a variational formula for the logarithm of the Laplace transform, see Sections 2 and 4 of \cite{LiSzni15}, but (\ref{1.40}) will suffice for our present purpose.

\medskip
We now turn to the description of the excursions in the interlacements that will be of interest in the next sections. We consider a box $U$ in $\IZ^d$ (see the beginning of this section) and $A \subseteq U$ a non-empty set. By definition of $\Omega$, see (\ref{1.36}), we know that $w(\wh{W}^*_A \times [0,u]) < \infty$, for all $u \ge 0$, but $\omega(\wh{W}^*_A \times \IR_+) = \infty$. Moreover, the labels $u_i$ that appear in the point measure $\omega$ are all distinct, and each $w^*_i$ that belongs to $\wh{W}^*_A$ only contains finitely many excursions from $A$ to $\partial U$ (recall that the bilateral trajectory corresponding to the $\IZ^d$-valued coordinates of an element of $\wh{W}$ only spends finite time in any finite subset of $\IZ^d$).

\medskip
Thus, given $\o = \sum_{i \ge 0} \delta_{(w^*_i,u_i)}$ in $\Omega$, we can rank the infinite sequence of excursions from $A$ to $\partial U$ by lexicographical order, first by increasing size of $u_i$ such that $w_i^* \in \wh{W}^*_A$, and then by order of appearance inside a given trajectory $w^*_i \in \wh{W}_A^*$. In this fashion we obtain a sequence of $\Gamma(U)$-valued random variables on $\Omega$
\begin{equation}\label{1.41}
Z^{A,U}_\ell (\o), \; \ell \ge 1,
\end{equation}
which describes the ordered sequence of excursions from $A$ to $\partial U$ in $\o \in \Omega$. We will also be interested in the number of excursions from $A$ to $\partial U$ at level $u \ge 0$ in $\o$, namely:
\begin{equation}\label{1.42}
\begin{split}
N^{A,U}_u (\o) = & \;\mbox{the total number of excursions from $A$ to $\partial U$ in all $w_i^*$}
\\[-1ex]
 & \;\mbox{such that $u_i \le u$ and $w_i^* \in \wh{W}^*_A$, if $\o = \sum_{i \ge 0} \delta_{(w^*_i,u_i)}$},
 \\
 = &\; \langle\mu_{A,u}, N_{A,U}\rangle (\o)\;\;\mbox{(in the notation of (\ref{1.37}), (\ref{1.30}))}.
\end{split}
\end{equation}

\n
We close this section with some facts concerning the strongly non-percolative regime of $\cV^u$. Although these results are not of direct use in the present work, they provide an instructive context for the introduction of the critical value $\overline{u}$ at the beginning of the next section. Recall that $u_{**}$ has been defined in (\ref{0.2}), and that $0 < u_* \le u_{**} < \infty$ (see \cite{PopoTeix13} and \cite{Szni12a}). Further, one knows (see the above references) that
\begin{equation}\label{1.43}
\mbox{for $u > u_{**}$, $\IP[0 \stackrel{\cV^u}{\longleftrightarrow} \partial B_L] \le c(u)\,e^{-c'(u) L^{c''}}$, for $L \ge 1$}
\end{equation}

\n
(actually, when $d \ge 4$, one can choose $c'' = 1$, and when $d = 3$, $c'' = \frac{1}{2}$ or any value in $(0,1)$, see  \cite{PopoTeix13}). By a union bound, translation invariance, and (\ref{1.43}), one then sees that in the notation of (\ref{0.3}), for any $M > 1$ and $u > u_{**}$, $\IP[B_N \stackrel{\cV^u}{\longleftrightarrow} S_N] \underset{N}{\longrightarrow} 0$, and as a result 
\begin{equation}\label{1.44}
\lim\limits_N \IP[A_N] = 1, \; \mbox{when $u > u_{**}$}.
\end{equation}
On the other hand, in the percolative regime, when $u < u_*$, $\cV^u$ contains an infinite component (see (\ref{0.1})) and $\IP[B_N \stackrel{\cV^u}{\longleftrightarrow} S_N] \underset{N}{\longrightarrow} 1$, so that
\begin{equation}\label{1.45}
\lim\limits_N \IP[A_N] = 0, \; \mbox{when $u < u_*$}.
\end{equation}
In the next section we will introduce a critical value $\ov{u}$ and make precise what we mean by the strongly percolative regime of $\cV^u$.

\section{Good boxes and the strongly percolative regime}
\setcounter{equation}{0}

In this section we first introduce a parameter $\ov{u}$ that will pin down what we mean by the {\it strongly percolative regime} of the vacant set of random interlacements. We then define a system of boxes. The basic side-length in their construction is $L$, and it will later be chosen of order $(N \log N)^{\frac{1}{d-1}}$ in Section 5, with $N$ having the same interpretation as in (\ref{0.3}). Next, we introduce a notion of {\it good box}, which, in the present set-up, plays a similar role to the notion of $\psi$-good box from Section 5 of \cite{Szni}, see also Remark \ref{rem2.2} below. The main result in this section is Theorem \ref{theo2.3}. It shows that when the three parameters entering the definition of a good box are smaller than $\ov{u}$, a box is good except on an event with super-polynomially decaying probability in $L$.

\medskip
We begin with the definition of the critical parameter $\ov{u}$. Given $u > v > 0$, we say that the vacant set of random interlacements {\it strongly percolates at levels $u, v$,} when for $B = [0,L)^d$
\begin{align}
\lim\limits_L  \mbox{\f $\dis\frac{1}{\log L}$} \log \IP\,\big[& \mbox{$\cV^u \cap B$ has no connected component of diameter}, \label{2.1}
\\[-1.5ex]
&\mbox{at least $\frac{L}{10}\big] = - \infty$}, \nonumber
\intertext{and for $B' = L e + B$, with $|e| = 1$, and $D = [-3L, 4L)^d$,}
\lim\limits_L  \mbox{\f $\dis\frac{1}{\log L}$} \log \IP\,[ &\mbox{there exist connected components of $B \cap \cV^u$ and  $B' \cap \cV^u$ of}\label{2.2} 
\\[-1ex]
&\mbox{diameter at least $\frac{L}{10}$, which are not connected in $D \cap \cV^v] = - \infty$}.\nonumber
\end{align}
We then define the critical value
\begin{equation}\label{2.3}
\begin{split}
\ov{u} = \sup\{ &\mbox{$s > 0$; the vacant set of random interlacements strongly}
\\[-0.5ex]
&\mbox{percolates at levels $u,v$, whenever $u > v$ lie in $(0,s)\}$}.
\end{split}
\end{equation}

\n
We then refer to $0 < u < \ov{u}$, as the strongly percolative regime of the vacant set of random interlacements (as we will soon see, $0 < \ov{u} \le u_*$, so the definition is not vacuous, and pins down a subset of the percolative regime $0 < u < u_*$). We will refer to estimates as in (\ref{2.1}) or (\ref{2.2}), as super-polynomial decay in $L$ of the probabilities under consideration.

\begin{remark}\label{rem1.1} \rm ~

\medskip\n
1) Note that when the vacant set of random interlacements strongly percolates at levels $u > v  > 0$, then the probability that there exist two connected components of $\cV^u \cap B$ with diameter at least $\frac{L}{10}$, which are not connected in $D \cap \cV^v$, has super-polynomial decay in $L$, as a direct consequence of (\ref{2.1}), (\ref{2.2}) (and with $B, D$ as in (\ref{2.1}), (\ref{2.2})).

\bigskip\n
2) When the vacant set of random interlacements strongly percolates at levels $u > v > 0$, it follows from a union bound that $\IP[B(0,L) \stackrel{\cV^v}{\mbox{\large $\nleftrightarrow$}} \partial_i B(0,2L)]$ has super-polynomial decay in $L$ (recall from the beginning of Section 1 that $\partial_i B(0,2L)$ stands for the internal boundary of $B(0,2L)$). 

\medskip
Actually, with the help of 1), when $\ov{u} > u > v > w > 0$, we can patch up crossings in $\cV^v$ from $B(0,2^k)$ to $\partial_i B(0,2^{k+1})$ in $\cV^w$ for $k \ge k_0$ (note that a crossing from $B(0,2^k)$ to $\partial_i B(0,2^{k+1})$ has diameter at least $2^k$, and a crossing from $B(0,2^{k+1})$ to $\partial_i B(0,2^{k+2})$ has diameter at least $2^{k+1}$, so that both $2^k$ and $2^{k+1}$ exceed $\frac{L}{10}$, when $L = 2^{k+3} + 1$, the side-length of $B(0,2^{k+2})$), and find that $\cV^w$ percolates with positive probability. By ergodicity $\cV^w$ percolates with probability $1$, and one thus finds that
\begin{equation}\label{2.4}
\ov{u} \le u_* ( \le u_{**}).
\end{equation}

\n
3) By Theorem 1.1 of \cite{DrewRathSapo14a} (see also \cite{Teix11}, when $d \ge 5$), one knows that there are constants $c, c' > 0$, such that for $0 < u < c'$,
\begin{equation}\label{2.5}
\left\{ \begin{array}{ll}
\limsup\limits_n \;\mbox{\f $\dis\frac{1}{n^c}$}\, \log \IP\,[& \!\!\!\!\mbox{the infinite connected component of $\cV^u$ does not meet} 
\\[-1ex]
&\!\!\!\!\!\mbox{$B(0,n)] \le -1$, and}
\\[2ex]
\limsup\limits_n \;\mbox{\f $\dis\frac{1}{n^c}$} \,\log \IP\,[&\!\!\!\!\mbox{there exist two connected subsets of $\cV^u \cap B(0,n)$ with}  
\\[-1ex]
&\!\!\!\!\mbox{with diameter at least $\frac{n}{10}$, which are not connected in}  
\\
&\!\!\!\!\mbox{$\cV^u \cap B(0,2n)] \le -1$}.  
\end{array}\right.
\end{equation}

\medskip\n
It is then straightforward to see that for such an $u$ in $(0,c')$ when $v \in (0,u)$, the vacant set of random interlacements strongly percolates at levels $u,v$ (note that in the notation of (\ref{2.2}), when $e = e_i$, with $e_i$ a vector of the canonical basis $e_1,\dots,e_d$ of $\IR^d$, $L e + B(0,L) \supseteq B \cup B'$ and $Le + B(0,2L) \subseteq D$, whereas $B(0,L) \supseteq B \cup B'$ and $B(0,2L) \subseteq D$, when $e = -e_i$). We thus see that
\begin{equation}\label{2.6}
\ov{u} > 0.
\end{equation}

\n
It is of course a natural question whether all the above critical values actually coincide, that is, whether $\ov{u} = u_* = u_{**}$. \hfill $\square$
\end{remark}

\medskip
We now introduce a system of boxes that will play an important role in the subsequent analysis. We consider positive integers
\begin{equation}\label{2.7}
L \ge 1 \;\;\mbox{and} \; \; K \ge 100.
\end{equation}

\n
We will be interested in the regime where $L$ tends first to infinity, and we will later let $K$ become large. In fact, in Section 5, we will choose $L$ of order $(N \log N)^{\frac{1}{d-1}}$, see (\ref{5.1}), where $N$ has the same meaning as in (\ref{0.3}), and will let $N$ tend to infinity, and operate with large values of $K$. We introduce the lattice
\begin{equation}\label{2.8}
\IL= L \IZ^d
\end{equation}
and the boxes in $\IZ^d$
\begin{equation}\label{2.9}
\begin{split}
B_0 = &\; [0,L)^d \subseteq D_0 \subseteq [- 3L,4L)^d \subseteq \check{D}_0 = [-4L, 5L)^d 
 \\
\subseteq &\;U_0 = [-KL + 1, K L - 1)^d  \subseteq  \check{U}_0 = [-(K+1) L + 1, (K+1) L - 1)^d
\\
 \subseteq &\; \check{B}_0 = [-(K+1)L, (K+1)L)^d
\end{split}
\end{equation}

\n
as well as the translates of these boxes at the various sites of $\IL$:
\begin{equation}\label{2.10}
\begin{split}
B_z =  z +B_0 \subseteq  & \;D_z = z + D_0 \subseteq \check{D}_z = z + \check{D}_0 \subseteq U_z = z + U_0
\\
 \subseteq & \; \check{U}_z = z + \check{U}_0 \subseteq \check{B}_z = z + \check{B}_0.
\end{split}
\end{equation}

\medskip\n
This is quite similar to the construction in Section 4 of \cite{Szni}, see (\ref{4.3}), (\ref{4.4}), but here we introduced the additional boxes $\check{D}_z$ and $\check{U}_z$. These boxes will play an important role in Sections 4 and 5, in order to simultaneously handle excursions $Z_\ell^{D,U}$, $\ell \ge 1$ and $Z_\ell^{D',U'}$, $\ell \ge 1$, when $D = D_z$, $U = U_z$, and $D' = D_{z'}$, $U' = U_{z'}$, with $z,z'$ neighbors in $\IL$, and bring to bear soft local time techniques to construct couplings with i.i.d. excursions. This feature is related to the fact that, in the present context, we do not have at our disposal a simple decomposition, as in the case of the Gaussian free field, which could be written as the sum of a local field $\psi^z(\cdot)$ vanishing outside $U_z$ and an harmonic field $h^z(\cdot)$, with good independence properties of the local fields as $z$ varies, see for instance Lemma 4.1 and Lemma 5.3 of \cite{Szni}.

\medskip
We now come to the notion of good box alluded to at the beginning of the section. Very often, for convenience, we will refer to the boxes $B_z$, $z \in \IL$, as $L$-boxes, and write $B, D, \check{D}, U, \check{U}$ in place of $B_z, D_z, \check{D}_z, U_z, \check{U}_z$ with $z \in \IL$, when no confusion arises. We also write $Z^D_\ell, \ell \ge 1$, as a shorthand notation for $Z^{D,U}_\ell, \ell \ge 1$, i.e.~the successive excursions from $D$ to $\partial U$ in the interlacements, see (\ref{1.41}). Moreover, when $t \ge 1$ is a real number, we write $Z^D_t$ in place of $Z^D_\ell$, with $\ell = [t]$. By range$(Z^D_\ell)$ we mean the set of points in $\IZ^d$ visited by $Z^D_\ell$. Given an $L$-box $B$ and $\alpha > \beta > \gamma > 0$, we say that $B$ is {\it good at levels $\alpha, \beta, \gamma$}, or {\it good$(\alpha, \beta, \gamma)$}, when
\begin{align}
&\mbox{$B \,\backslash$(range $Z_1^D \cup \ldots \cup$ range $Z^D_{\alpha \,{\rm cap}(D)}$) contains a connected set with diameter}\label{2.11} 
\\[-0.5ex]
&\mbox{at least $\frac{L}{10}$ (and the set in parenthesis is empty when $\alpha\,{\rm cap}(D) < 1)$}, \nonumber
\\[2ex]
&\mbox{for any neighboring $L$-box $B' = L \,e + B$ of $B$ (i.e. $|e| = 1$), any two connected} \label{2.12}
\\[-0.5ex]
&\mbox{sets with with diameter at least $\frac{L}{10}$ in $B \backslash ($range $Z^D_1 \cup \ldots \cup$ range $Z^D_{\alpha\,{\rm cap}(D)}$) and} \nonumber
\\[-0.5ex]
&\mbox{$B' \backslash($range $Z^{D'}_1 \cup \ldots \cup$ range $Z^{D'}_{\alpha \,{\rm cap}(D')}$) are connected in} \nonumber
\\[-0.5ex]
&\mbox{$D \backslash($range $Z^{D}_1 \cup \ldots \cup$ range $Z^{D}_{\beta \,{\rm cap}(D)}$) (with a similar convention as in (\ref{2.11}))} \nonumber
\intertext{(note that ${\rm cap}(D) = {\rm cap}(D')$ by translation invariance),}
&\dsl_{1 \le \ell \le \beta\,{\rm cap}(D)} \dil^{T_U}_0 e_D \big(Z^D_\ell(s)\big)\,ds \ge \gamma \,{\rm cap}(D)\label{2.13} 
\end{align}

\medskip\n
(we recall that $e_D$ stands for the equilibrium measure of $D$, see (\ref{1.6})).

\medskip
When $B$ is not good at levels $\alpha, \beta, \gamma$, we say that it is {\it bad$(\alpha, \beta, \gamma)$}.

\begin{remark}\label{rem2.2} \rm Informally, (\ref{2.11}), (\ref{2.12}) play the role of conditions (5.7) and (5.8) of \cite{Szni}, for the definition of $\psi$-good box at levels $\alpha, \beta$. In the context of \cite{Szni} the excursion sets $\{\psi^B \ge \alpha\} \cap B$, $\{\psi^{B'} \ge \alpha\} \cap B$, and $\{\psi^B \ge \beta\} \cap D$, which involve the local fields $\psi^B$ and $\psi^{B'}$, are now replaced in the present set-up by $B \backslash ({\rm range}\, Z^D_1 \cup \ldots \cup {\rm range}\,Z^D_{\alpha\,{\rm cap} (D)}$), $B' \backslash ({\rm range}\, Z^D_1 \cup \ldots \cup \,{\rm range}\,Z^{D'}_{\alpha\,{\rm cap} (D')}$), and $D \backslash ({\rm range}\, Z^D_1 \cup \ldots \cup\, {\rm range}\,Z^D_{\beta\,{\rm cap} (D)}$), and involve the excursions $Z^D_\ell$ and $Z^{D'}_{\ell'}$, with $\ell \le \alpha\, {\rm cap}(D)$ (or $\ell \le \beta \,{\rm cap}(D)$) and $\ell' \le \alpha \,{\rm cap}(D')$. The condition (\ref{2.13}) has no equivalent in \cite{Szni}, and will have an important consequence in Section 3, where the controls of Theorem  \ref{theo3.2} play a similar role to Corollary 4.4 of \cite{Szni}. 

\hfill $\square$
\end{remark}

\medskip
We introduce one more notation. Given $z \in \IL$, $D = D_z$, $U = U_z$ and $u \ge 0$, we write the number of excursions from $D$ to $\partial U$ in the interlacement at level $u$ as
\begin{equation}\label{2.14}
N_u(D) = N_u^{D,U} \quad \mbox{(see (\ref{1.42}))}.
\end{equation}

\n
We are now ready to state and prove the main result of this section. It plays an analogous role in the present work to Proposition 5.2 of \cite{Szni}.

\begin{theorem}\label{theo2.3} (super-polynomial decay of being bad at levels below $\ov{u}$)

\medskip
For any $\alpha > \beta > \gamma$ in $(0,\ov{u})$ and $K \ge c_4(\alpha,\beta,\gamma) (\ge 100)$, one has
\begin{equation}\label{2.15}
\lim\limits_L \;\mbox{\f $\dis\frac{1}{\log L}$} \log \IP[B\;\mbox{is bad$(\alpha, \beta, \gamma)] = - \infty$}
\end{equation}
(the probability in (\ref{2.15}) does not depend on the specific choice of the $L$-box $B$).
\end{theorem}

\begin{proof}
We first control the probability that (\ref{2.11}) does not hold. We pick $u_0,v_0$ so that 
\begin{equation}\label{2.16}
\ov{u} > u_0 > \alpha > v_0 > \beta .
\end{equation}

\n
We observe that when $N_{u_0}(D) \ge \alpha \,{\rm cap}(D)$, then $B \backslash ({\rm range} \, Z^D_1 \cup \dots \cup {\rm range}\,Z^D_{\alpha\,{\rm cap}(D)}$) contains $B \cap \cV^{u_0}$ and therefore
\begin{equation}\label{2.17}
\begin{split}
\IP[\mbox{(\ref{2.11}) does not hold}] \le &\;  \IP[N_{u_0}(D) < \alpha \,{\rm cap}(D)] \; +
\\
 &\; \IP[B \cap \cV^{u_0} \;\mbox{has no component of diameter $\ge \frac{L}{10}]$}.
\end{split}
\end{equation}

\n
By the choice of $u_0$ in (\ref{2.16}), and (\ref{2.1}), the last expression has super-polynomial decay in $L$. On the other hand, the number $N_{u_0}(D)$ of excursions from $D$ to $\partial U$ at level $u_0$ is at least $\mu_{D,u_0}(\wh{W}_+)$, a Poisson variable with parameter $u_0 \, {\rm cap}(D)$ (see (\ref{1.38})). So, for $\lambda > 0$, by the exponential Chebyshev inequality we have
\begin{equation}\label{2.18}
\begin{split}
\IP[N_{u_0}(D) \le \alpha \, {\rm cap}(D)] \le &\;  \IP[\langle \mu_{D,u_0},1\rangle \le \alpha\,{\rm cap}(D)]
\\
 \le &\;\exp\big\{- u_0 \,{\rm cap}(D) \big(1 - e^{-\lambda} - \mbox{\f $\dis\frac{\alpha}{u_0}$} \;\lambda\big)\big\}.
\end{split}
\end{equation}
Since $\frac{\alpha}{u_0} < 1$, we can pick $\lambda$ small and ensure that the expression in parenthesis in the last line is strictly positive. By (\ref{1.8}), ${\rm cap}(D) \ge c L^{d-2}$, and it follows from the above that
\begin{equation}\label{2.19}
\mbox{$\IP[$(\ref{2.11}) does not hold$]$ has super-polynomial decay in $L$}.
\end{equation}

\n
We now control the probability that (\ref{2.12}) fails. We pick $u_1,v_1$ such that
\begin{equation}\label{2.20}
\ov{u} > \alpha > u_1 > v_1 > \beta, \;\;\mbox{for instance $u_1 = \mbox{\f $\dis\frac{3 \alpha + \beta}{4}$}, \;\; v_1 = \mbox{\f $\dis\frac{\alpha + 3 \beta}{4}$}$}.
\end{equation}
Then, consider $B$ and $B'$ neighboring $L$-boxes. We note that when $N_{u_1}(D) \le \alpha \,{\rm cap}(D)$, $N_{u_1}(D') \le \alpha \,{\rm cap}(D)$ (recall that $N_{u_1}(D') = N_{u_1}^{D',U'}$ and ${\rm cap}(D') = {\rm cap}(D)$), and $N_{v_1}(D) \ge \beta \,{\rm cap}(D)$, then
\begin{equation*}
\begin{array}{l}
B \backslash ({\rm range} \,Z^D_1 \cup \ldots \cup {\rm range} \,Z^D_{\alpha \,{\rm cap}(D)}) \subseteq B \cap \cV^{u_1},
\\[1ex]
B' \backslash ({\rm range} \,Z^{D'}_1 \cup \ldots \cup {\rm range} \,Z^{D'}_{\alpha \,{\rm cap}(D')}) \subseteq B \cap \cV^{u_1}, \; \mbox{and}
\\[1ex]
D \backslash ({\rm range} \,Z^D_1 \cup \ldots \cup {\rm range} \,Z^D_{\beta \,{\rm cap}(D)}) \supseteq D \cap \cV^{v_1}.
\end{array}
\end{equation*}
As a result we see that
\begin{equation}\label{2.21}
\begin{split}
\IP\,[ &\mbox{(\ref{2.12}) does not hold}] \le \IP[N_{u_1}(D) >\alpha \,{\rm cap}(D)] + \IP[N_{u_1}(D') > \alpha\,{\rm cap}(D)] 
\\[-0.5ex]
& + \IP[N_{v_1}(D) < \beta \,{\rm cap}(D)] + 
\\
 \IP\,[& \mbox{there exist connected sets  of $B \cap \cV^{u_1}$ and $B' \cap \cV^{u_1}$ of diameter $\ge \frac{L}{10}$,}
\\[-0.5ex]
& \mbox{which are not connected in $D \cap \cV^{v_1}]$}.
\end{split}
\end{equation}

\medskip\n
Inside the last probability we can replace ``connected sets'' by ``connected components'' without changing the event, since $\cV^{u_1} \subseteq \cV^{v_1}$. So, by (\ref{2.2}) the last term of (\ref{2.21}) has super-polynomial decay in $L$. By a similar argument as in (\ref{2.18}), since $v_1 > \beta$, the third probability in the right-hand side of (\ref{2.21}) also has super-polynomial decay in $L$. By translation invariance the first two probabilities in the right-hand side of (\ref{2.21}) are equal. We bound them as follows. For $\lambda > 0$, we have by the exponential Chebyshev inequality
\begin{equation}\label{2.22}
\begin{array}{ll}
\IP\,[N_{u_1}(D) > \alpha\,{\rm cap}(D)] &\!\!  \le \; \exp\{ - \lambda \,\alpha \,{\rm cap}(D)\}\; \IE[e^{\lambda N_{u_1}(D)}]
\\[1ex]
&\!\!\!\!\!\!\!\!\!\! \stackrel{(\ref{1.42}),(\ref{1.38})}{=} \exp\{{\rm cap}(D)(- \lambda \alpha + u_1 \,E_{\ov{e}_D} [e^{\lambda N_{D,U}}- 1])\},
\end{array}
\end{equation}
where we recall the notation of (\ref{1.30}).

\medskip
By Lemma \ref{lem1.7}, with $p$ as above (\ref{1.31}), we see that for $\lambda  < \wt{c}$, $K > c'$,
\begin{equation*}
E_{\ov{e}_D} [e^{\lambda N_{D,U}} - 1] \le \mbox{\f $\dis\frac{e^\lambda p}{1 - e^\lambda (1-p)}$} - 1 = \mbox{\f $\dis\frac{e^\lambda -1}{1 - e^\lambda (1-p)}$} \le \mbox{\f $\dis\frac{e^\lambda -1}{1 - e^\lambda \, cK^{2-d}}$} \;,
\end{equation*}
where we used the lower bound on $p$ in (\ref{1.31}), as well as the upper bound on ${\rm cap}(D)$ in (\ref{1.8}), and the lower bound $d(D,U^c) \ge c KL$ from (\ref{2.9}), (\ref{2.10}). For $K \ge c(\alpha, \beta)$ (recall that  $u_1 = \frac{3 \alpha + \beta}{4} < \alpha$), we see that
\begin{equation*}
(1 - e^{\wt{c}} cK^{2-d})^{-1} \le \fr \,\Big(\mbox{\f $\dis\frac{\alpha}{u_1}$} + 1\Big),
\end{equation*}
and choosing $\lambda > 0$ small enough, we have
\begin{equation*}
- \lambda \alpha + u_1 \,E_{\ov{e}_D} [e^{\lambda N_{D,U}} - 1] < 0.
\end{equation*}
Inserting the lower bound for ${\rm cap}(D)$ from (\ref{1.8}) in (\ref{2.22}) we see that
\begin{equation}\label{2.23}
\mbox{when $K \ge c(\alpha,\beta)$, $\IP[$(\ref{2.12}) does not hold$]$ has super-polynomial decay in $L$}.
\end{equation}
We finally bound the probability that (\ref{2.13}) fails. We now choose $u_2$ such that
\begin{equation}\label{2.24}
\beta > u_2 > \gamma, \;\mbox{for instance $u_2 = \mbox{\f $\dis\frac{\beta + \gamma}{2}$}$}\,.
\end{equation}
We observe that when $N_{u_2}(D) \le \beta \,{\rm cap}(D)$, then
\begin{equation}\label{2.25}
\begin{split}
\big\langle \mu_{D,u_2}(dw),  \dil^\infty_0  e_D\big(X_s(w)\big)ds \big\rangle & =  \dsl_{1 \le \ell \le N_{u_2}(D)} \,\dil^{T_U}_0 e_D \big(Z^D_\ell(s)\big)ds
\\
 & \le  \dsl_{1 \le \ell \le \beta \, {\rm cap}(D)} \,\dil^{T_U}_0 e_D \big(Z^D_\ell(s)\big)ds .
\end{split}
\end{equation}
It now follows that for $\lambda > 0$,
\begin{equation}\label{2.26}
\begin{split}
\IP[\mbox{(\ref{2.13}) does not hold}]  \le& \; \IP[N_{u_2}(D) > \beta \,{\rm cap}(D)] \;+
\\
& \;\IP\big[\big\langle \mu_{D,u_2}, \dil^\infty_0 e_D(X_s) ds\big\rangle < \gamma \,{\rm cap} (D)\big].
\end{split}
\end{equation}

\n
Since $u_2 = \frac{\beta + \gamma}{2} < \beta$, we see as below (\ref{2.22}) that for $K \ge c(\beta,\gamma)$, the first term on the right-hand side of (\ref{2.26}) has super-polynomial decay in $L$ (actually, even exponential decay at rate $L^{d-2}$). As for the last term of (\ref{2.26}), we see by (\ref{1.38}) and Lemma \ref{lem1.1}, that for $\lambda > 0$ it is smaller than 
\begin{equation*}
\exp\{{\rm cap}(D) (\lambda \gamma + u_2 \,E_{\ov{e}_D} [e^{-\lambda \int^\infty_0 e_D(X_s)ds} - 1])\} \stackrel{(\ref{1.13})}{=} \exp \Big\{{\rm cap}(D) \,\lambda \Big(\gamma - \mbox{\f $\dis\frac{u_2}{1 + \lambda}$}\Big)\Big\}.
\end{equation*}

\n
Since $u_2 = \frac{\beta + \gamma}{2} > \gamma$, we can choose $\lambda$ small so that the last expression in parenthesis is negative. Taking into account the lower bound on ${\rm cap}(D)$ from (\ref{1.8}) the term in the last line of (\ref{2.26}) has exponential decay at rate $L^{d-2}$. We have in particular shown that
\begin{equation}\label{2.27}
\mbox{when $K \ge c(\beta, \gamma)$, $\IP[$(\ref{2.13}) does not hold$]$ has super-polynomial decay in $L$}.
\end{equation}
Combining (\ref{2.19}), (\ref{2.23}), and (\ref{2.27}), Theorem \ref{theo2.3} follows.
\end{proof}

\section{Occupation-time bounds}
\setcounter{equation}{0}

In this section we derive uniform controls on the probability that simultaneously, in a finite collection of $L$-boxes, which are all good at levels $\alpha,\beta,\gamma$ belonging to $(u,\ov{u})$, and well spread-out, the occupation numbers of all corresponding $D$-boxes exceed $\beta \,{\rm cap}(D)$, see Theorem \ref{theo3.2}. These estimates will later play an important role in the proof of the central claim (\ref{0.6}), see Theorem \ref{theo6.3}. In spirit, the results of this section are similar to the upper bounds derived in Section 4 of \cite{Szni}, in the context of the level-set percolation of the Gaussian free field. There, in Corollary 4.4 of \cite{Szni}, uniform upper bounds are derived on the probability that simultaneously, in a finite collection of $B$-boxes, the harmonic averages of the Gaussian free field attached to much larger concentric $U$-type boxes, reach a value below a certain negative level $-a$ in the $B$-boxes. In the present context, the corresponding condition is that the occupation numbers of the $D$-boxes at level $u$ each exceed $\beta \,{\rm cap}(D)$. The additional constraint that the $B$-boxes are good$(\alpha, \beta,\gamma)$ permits, by (\ref{2.13}), to translate the information on the occupation numbers at level $u$, into an information on the total occupation-time for the interlacement at level $u$ in each $D$-box (of the collection). Whereas Gaussian bounds (such as the Borell-TIS inequality and the Dudley inequality) were central tools in Section 4 of \cite{Szni}, here instead, we are guided by some insights on the Laplace transform of the occupation-times of random interlacements and its link to large deviations of the occupation-time profiles of random interlacements, gained from \cite{LiSzni15}, see Remark \ref{rem3.3} below.

\medskip
We keep the notation of the previous section, and recall that $L \ge 1$ and $K \ge 100$ are positive integers. We consider (see (\ref{2.8}) for notation)
\begin{equation}\label{3.1}
\begin{array}{l}
\mbox{$\cC$ a non-empty finite subset of $\IL$ with points at mutual $|\cdot|_\infty$-distance at least $\ov{K}L$,}
\\
\mbox{where $\ov{K} = 2 K + 3$}.
\end{array}
\end{equation}

\n
Note that when $z_1 \not= z_2$ belong to $\cC$, then the corresponding $\check{B}_{z_1}$, $\check{B}_{z_2}$ (see (\ref{2.10})) satisfy
\begin{equation}\label{3.2}
d(\check{B}_{z_1}, \check{B}_{z_2}) \ge L .
\end{equation}
For a given $\cC$ as above, we define
\begin{equation}\label{3.3}
C = \textstyle \bigcup\limits_{z \in \cC} D_z .
\end{equation}
We will sometimes say that $D$ is in $\cC$ to mean $D = D_z$, with $z \in \cC$, and write $C = \bigcup_{D \in \cC} D$ to mean (\ref{3.3}). With a similar convention we introduce the function on $\IZ^d$ (see (\ref{1.6}), (\ref{1.9}) for notation) 
\begin{equation}\label{3.4}
V(x) = \dsl_{D \in \cC} e_C (D) \,\ov{e}_D(x), \;\; x \in \IZ^d
\end{equation}

\n
(where $e_C(D) = \sum_{y \in D} e_C(y)$, with $e_C$ the equilibrium measure of $C$). As a step towards the main goal of this section, namely Theorem \ref{theo3.2}, we first prove a proposition that compares $GV$ to the equilibrium potential $G e_C$ (denoted by $h_C$).

\begin{proposition}\label{prop3.1} (see (\ref{1.3}) for notation)

\medskip
Consider $\ve \in (0,1)$. Then, for any $L \ge 1$, $K \ge c_5(\ve)$, and $\cC$ as in (\ref{3.1}) one has
\begin{equation}\label{3.5}
\mbox{$GV \le (1 + \ve) h_C$, where $h_C = G e _C$ is the equilibrium potential of $C$ (see (\ref{1.10}))}.
\end{equation}
\end{proposition}

\begin{proof}
The claim (\ref{3.5}) will follow once we show that
\begin{equation}\label{3.6} 
V \le (1 + \ve) \,e_C.
\end{equation}

\n
Both $V$ and $e_C$ are supported by $C$ and it suffices to show that for any $z \in \cC$, and $D = D_z$, we have
\begin{equation}\label{3.7}
e_C(D) \,\overline{e}_D(x) \le (1 + \ve) \,e_C(x), \; \mbox{for all} \; x \in D( \subseteq C).
\end{equation}

\n
We can apply Proposition \ref{prop1.5} with the choice $A = D-z$, $B = C-z$ (recall $D = D_z)$, and note that by (\ref{3.1}), (\ref{3.2}), $\IZ^d \backslash (C-z)$ is connected). Thus, by (\ref{1.26}) we see that when $K \ge c_5(\ve)$, (\ref{3.7}) holds. This concludes the proof of Proposition \ref{prop3.1}.
\end{proof}

\medskip
We are now ready for the main result of this section. It plays the role of Corollary 4.4 of \cite{Szni}. The notion of $h$-good box, see (5.9) of \cite{Szni}, corresponds to condition $N_u(D) \ge \beta \,{\rm cap}(D)$ in the present context (see (\ref{2.14}) for notation).

\begin{theorem}\label{theo3.2} ($\ov{u}$ as in (\ref{2.3}), see (\ref{2.11}) - (\ref{2.13}) for the definition of good$(\alpha, \beta, \gamma)$)

\medskip
Consider $0 < u < \ov{u}$, $0 < \ve < 1$ such that $\ve( \mbox{\f $\sqrt{\mbox{\normalsize $\frac{\ov{u}}{u}$}}$} - 1) < 1$, and $\alpha > \beta > \gamma$ in $(u,\ov{u})$. Then, for $K \ge c_5(\ve)$, $L \ge 1$, and $\cC$ as in (\ref{3.1}) we have
\begin{equation}\label{3.15}
\begin{array}{l}
\IP \big[\bigcap\limits_{z \in \cC} \{B_z \;\mbox{is good$(\alpha, \beta, \gamma)$ and $N_u(D_z) \ge \beta \,{\rm cap}(D)\}\big] \le$}
\\[2ex]
\exp \big\{ - \big(\sqrt{\gamma} - \frac{\sqrt{u}}{1- \ve(\sqrt{\frac{\ov{u}}{u}} - 1)} \,\big) (\sqrt{\gamma} - \sqrt{u}) \,{\rm cap}(C)\big\}.
\end{array}
\end{equation}
\end{theorem}

\begin{proof}
When $B = B_z$, with $z$ in $\cC$, is good$(\alpha, \beta, \gamma)$ and $N_u(D) \ge \beta\,{\rm cap}(D)$ (with $D = D_z)$, then, in the notation of (\ref{1.37})
\begin{equation}\label{3.16}
\begin{array}{lcl}
\big\langle \mu_{C,u}, \dil^\infty_0 e_D(X_s) ds \big\rangle &\!\!\! \stackrel{D \subseteq C}{=} &\!\!\! \big\langle \mu_{D,u}, \dil^\infty_0 e_D(X_s) ds \big\rangle
\\[2ex]
&\!\!\!  \ge &\!\!\!  \dsl^{\beta\, {\rm cap}(D)}_{\ell = 1} \dil^{T_U}_0 e_D(Z^D_\ell(s)) ds \quad \mbox{(since $N_u(D) \ge \beta \,{\rm cap}(D)$)}
\\[0.5ex]
& \!\!\! \stackrel{(\ref{2.13})}{\ge} & \!\!\! \gamma\, {\rm cap}(D).
\end{array}
\end{equation}
Thus, with $V$ as in (\ref{3.4}), we find that
\begin{equation}\label{3.17}
\begin{array}{l}
\IP\big[\bigcap\limits_{z \in \cC} \big\{B_z \; \mbox{is good$(\alpha, \beta, \gamma)$ and $N_u(D_z) \ge \beta \,{\rm cap}(D)\big\}\big] \le$} 
\\[1ex]
\IP\big[\bigcap\limits_{D \:{\rm in}\; \cC} \big\{\big\langle \mu_{C,u}, \dil^\infty_0 e_D(X_s) ds\big\rangle \ge \gamma \,{\rm cap}(D)\big\}\big] \stackrel{(\ref{3.4})}{\le} 
\\[2ex]
\IP\big[\big\langle \mu_{C,u}, \dil^\infty_0 V(X_s)\big\rangle ds \ge \gamma \,{\rm cap}(C)\big],
\end{array}
\end{equation}
where in the last step we have used that $\sum_{D \,{\rm in} \, \cC} e_C(D) = e_C(C) = {\rm cap}(C)$. Setting $a = (\sqrt{\gamma} - \sqrt{u}) / \sqrt{\gamma} \in (0,1)$, we then see that the last probability in (\ref{3.17}) equals
\begin{equation}\label{3.18}
\begin{array}{l}
\IP \big[\big\langle \mu_{C,u}, \dil^\infty_0 a \,V(X_s)ds\big\rangle \ge (\sqrt{\gamma} - \sqrt{u}) \,\sqrt{\gamma} \,{\rm cap}(C)\big] \stackrel{\rm Chebyshev}{\le}
\\[1ex]
\exp\{ - \sqrt{\gamma} (\sqrt{\gamma} - \sqrt{u}) \,{\rm cap}(C)\} \;\IE\Big[\exp\Big\{\dsl_{x \in \IZ^d} a\,V(x) \,L_{x,u}\Big\}\Big].
\end{array}
\end{equation}

\n
When $K \ge c_5(\ve)$, we know by Proposition \ref{prop3.1} that $\|G aV\|_\infty \le a(1 + \ve) = 1 -  \mbox{\f $\sqrt{\mbox{\normalsize $\frac{u}{\gamma}$}}$}  + \ve(1 - \mbox{\f $\sqrt{\mbox{\normalsize $\frac{u}{\gamma}$}}$} ) < 1$, since by assumption $\ve( \mbox{\f $\sqrt{\mbox{\normalsize $\frac{\gamma}{u}$}}$} -1) \le \ve ( \mbox{\f $\sqrt{\mbox{\normalsize $\frac{\ov{u}}{u}$}}$} - 1) < 1$. We then see that $\|(I - GaV)^{-1} 1\|_\infty \le (1 - a(1 + \ve))^{-1}$. Hence, by (\ref{1.40}), the last expectation in (\ref{3.18}) equals
\begin{equation}\label{3.19}
\begin{array}{l}
\exp\{u \langle a V, (I-GaV)^{-1}1\rangle\} \le \exp\big\{u \,a(1-a(1 + \ve)\big)^{-1} \langle V,1 \rangle\big\} =
\\[1ex]
\exp\Big\{\sqrt{u} \Big(1 - \ve \Big(\sqrt{\mbox{\f $\dis\frac{\gamma}{u}$}} - 1 \Big)\Big)^{-1} (\sqrt{\gamma} - \sqrt{u}) \,{\rm cap}(C)\Big\} \le 
\\[1ex]
\exp\Big\{\sqrt{u} \Big(1 - \ve \Big(\sqrt{\mbox{\f $\dis\frac{\ov{u}}{u}$}} - 1 \Big)\Big)^{-1} (\sqrt{\gamma} - \sqrt{u}) \,{\rm cap}(C)\Big\} ,
\end{array}
\end{equation}

\n
where we used the identity $\langle V,1\rangle = {\rm cap}(C)$ (see (\ref{3.4})) and the definition of $a$.
Inserting this bound in the expectation in the last line of (\ref{3.18}), and then coming back to (\ref{3.17}), we obtain (\ref{3.15}). This concludes the proof of Theorem \ref{theo3.2}.
\end{proof}

\begin{remark}\label{rem3.3} \rm
Informally, the above proof follows the spirit of \cite{LiSzni15}. The large deviation results of \cite{LiSzni15} suggest the following procedure. To control the occurrence of a bump over $C$ at level $\gamma$ ($> u$) for the occupation-times of random interlacements, one applies  the Chebyshev inequality to the Laplace transform of the occupation-times tested against the function $W = -\frac{Lf}{f}$, where $L$ stands for the generator of the continuous-time simple random walk, and \hbox{$f = (\sqrt{\frac{\gamma}{u}} - 1) \,h_C + 1$,} i.e. $W = a \, e_C$, and $a$ in the notation of (\ref{3.18}) equals $1 - \mbox{\scriptsize $\sqrt{\frac{u}{\gamma}}$}$. In place of $W$, we instead use the approximation $a\,V$, with $V$ as in (\ref{3.4}), which is better suited to cope with the precise type of bump in occupation-times we are faced with \hbox{in (\ref{3.16}).} 

Later on, the coarse graining procedure developed in the next sections will in essence show that the disconnection of $B_N$ by the random interlacements at level $u$ ($< \ov{u}$) is well controlled, in principal exponential order, by the occurrence of a bump in occupation-times, at a level $\gamma$ arbitrarily close to $< \ov{u}$, over some $C$ as above, which ``surrounds well'' $B_N$, see (\ref{6.8})-(\ref{6.10}) below. In close spirit, the occurrence of bumps in the density profile of the occupation-times of random interlacements, insulating a macroscopic body, were investigated in \cite{LiSzni15}, see for instance (0.7) or Theorems 6.2 and 6.4 of \cite{LiSzni15}.
\hfill $\square$
\end{remark}

\section{Coupling of excursions}
\setcounter{equation}{0}

In this section we prepare the ground for the next section. We introduce a coupling based on the soft local time techniques developed in \cite{PopoTeix13} and in \cite{ComeGallPopoVach13}, between the excursions $Z^{\check{D},\check{U}}_k$, $k \ge 1$, from $\check{D}$ to $\partial \check{U}$ in the random interlacements, where $\check{D},\check{U}$ runs over a collection $\check{D}_z$, $\check{U}_z$, $z \in \cC$, with $\cC$ a finite subset of $\IL$ as in (\ref{3.1}), and a certain collection of i.i.d. excursions $\wt{Z}_k^{\check{D}}$, $k \ge 1$, from $\check{D}$ to $\partial \check{U}$, which are independent as $\check{D},\check{U}$ runs over the same $\check{D}_z, \check{U}_z$, $z \in \cC$ as above. This coupling will play an important role in the next section, when proving the super-exponential estimate on the probability of finding more than a few columns containing a bad box in Theorem \ref{theo5.1}. The main result of this section is Proposition \ref{prop4.1}. It contains the crucial controls on the coupling of excursions that we will need in the next section.

\medskip
As in Sections 2 and 3, we consider integers $L \ge 1$ and $K \ge 100$, as well as a non-empty subset $\cC$ of $\IL$ ($= L \IZ^d)$ satisfying (\ref{3.1}). We recall that the boxes $\check{B}_z$, $z \in \cC$ are at pairwise mutual $|\cdot |_\infty$-distance at least $L$, see (\ref{3.2}). A central object of interest for us in this section are the excursions $Z^{\check{D}_z,\check{U}_z}_k$, $k \ge 1$, from $\check{D}_z$ to $\partial \check{U}_z$, in the interlacements (see (\ref{1.41})), as $z$ runs over $\cC$. As a shorthand we will simply write $Z^{\check{D}}_k$, $k \ge 1$, in place of $Z^{\check{D}_z,\check{U}_z}_k$, $k \ge 1$ (where $\check{D},\check{U}$ runs over $\check{D}_z$, $\check{U}_z$, $z \in \cC$). This is in line with the similar shorthand notation $Z^D_\ell$, $\ell \ge 1$, for the excursions from $D$ to $\partial U$, introduced above (\ref{2.11}). The reason for the introduction of the $Z^{\check{D}}_k$, $k \ge 1$, is the fact that this sequence contains the information of both sequences $Z^D_\ell$, $\ell \ge 1$ and $Z^{D'}_\ell$, $\ell \ge 1$, when $D'$ is a ``neighboring'' box of $D$, that is, when $D = D_z$, $D' = D_{z'}$, $\check{D} = \check{D}_z$, with $z \in \cC$ and $z'$ in $\IL$ is a neighbor of $z$. Indeed, $D$ and $D'$ are contained in $\check{D}$ and $U$ and $U'$ are contained in $\chu$ by (\ref{2.9}), (\ref{2.10}). This feature will be important when using the coupling constructed in this section to bound the probability that collections of $L$-boxes $B$ are bad$(\alpha, \beta, \gamma)$, since this last constraint, via the negation of (\ref{2.12}), simultaneously involves the excursions $Z^D_\ell$, $\ell \ge 1$, and the excursions $Z^{D'}_\ell$, $\ell \ge 1$.

\medskip
The soft local time technique of \cite{PopoTeix13} offers a way to couple the excursions $Z^{\check{D}_z}_k$, $k \ge 1$, $z \in \cC$, of the random interlacements, with independent excursions $\wt{Z}^{\check{D}_z}_k$, $k \ge 1$, $z \in \cC$. We refer to Section 2 of \cite{ComeGallPopoVach13} for details. We will use here some facts that we recall below. First some notation. We introduce the subsets of $\IZ^d$
\begin{equation}\label{4.1}
\check{C} = \textstyle \bigcup\limits_{z \in \cC} \;\check{D}_z \subseteq \check{V} = \bigcup\limits_{z \in \cC} \check{U}_z ,
\end{equation}

\n
and, for any $x \in \IZ^d$, denote by $Q_x$ the probability measure governing two independent continuous-time walks $X^1_{^\point}$ and $X^2_{^\point}$ on $\IZ^d$ respectively starting from $x$ and from the initial distribution $\ov{e}_{\check{C}}$ (the normalized equilibrium measure of $\check{C}$). We consider the random variable $Y$ defined as 
\begin{equation}\label{4.2}
Y = \left\{ \begin{array}{ll}
X^1_{H_{\check{C}}}, & \mbox{on $\{H_{\check{C}} < \infty\}$} \quad \mbox{(with $H_{\check{C}}$ the entrance time of $X^1$ in $\check{C}$)}
\\[1ex]
X^2_0, &  \mbox{on $\{H_{\check{C}} = \infty\}$}.
\end{array}\right.
\end{equation}

\n
With the soft local time technique, one constructs a coupling $\IQ^\cC$ of the law $\IP$ of the random interlacements with a collection of independent right-continuous, Poisson counting functions, with unit intensity, vanishing at $0$, $(n_{\chd_{z}}(0,t))_{t \ge 0}$, $z \in \cC$, and with independent collections of i.i.d. excursions $\wt{Z}^{\check{D}_z}_k$, $k \ge 1$, $z \in \cC$, having for each $z \in \cC$ the same law on $\Gamma (\check{U}_z)$ as $X_{\cdot \, \Lambda T_{\check{U}_z}}$ under $P_{\ov{e}_{\check{D}_z}}$ (recall $\ov{e}_{\check{D}_z}$ stands for the normalized equilibrium measure of $\check{D}_z$). 
%
For convenience, we also set (for $\check{D} = \check{D}_z$, with $z \in \cC$)
\begin{equation}\label{4.3}
n_{\chd}(a,b) = n_{\chd}(0,b) - n_{\chd}(0,a), \;\mbox{for $0 \le a \le b$}
\end{equation}
(the notation is consistent when $a=0$).Thus, one has
\begin{equation}\label{4.4}
\begin{array}{l}
\mbox{under $\IQ^\cC$, as $\chd$ varies over $\chd_z$, $z \in \cC$, the ($(n_{\chd}(0,t))_{t \ge 0}$, $\wt{Z}^{\chd}_k, k \ge 1$) are}
\\
\mbox{independent collections of independent processes, with $(n_{\chd}(0,t))_{t \ge 0}$}
\\
\mbox{distributed as a Poisson counting process of intensity $1$, and $\wt{Z}^{\chd}_k, k \ge 1$,  }
\\
\mbox{as i.i.d. $\Gamma(\chu)$-variables with same law as $X_{\cdot\, \wedge T_{\chu}}$ under $P_{\ov{e}_{\chd}}$}.
\end{array}
\end{equation}    

\n
In addition, the coupling governed by $\IQ^\cC$ has the following crucial property, see Lemma 2.1 of \cite{ComeGallPopoVach13}. If for some $\delta \in (0,1)$, and all $\chd = \chd_z$, $z \in \cC$, $y \in \chd$, $x \in \partial \check{V}$
\begin{equation}\label{4.5}
\big(1 - \mbox{\f $\dis\frac{\delta}{3}$}\big) \ov{e}_{\chd}(y) \le Q_x[Y = y\,|\,Y \in \chd] \le \big(1 + \mbox{\f $\dis\frac{\delta}{3}$}\big) \,\ov{e}_\chd(y),
\end{equation}
then, on the event (in the auxiliary space governed by $\IQ^\cC$, and with $m_0 \ge 1$ arbitrary)
\begin{equation}\label{4.6}
\begin{split}
\wt{U}^{m_0}_\chd = \{ & n_{\chd} (m,(1 + \delta) m) < 2 \delta m,\;(1-\delta) \,m < n_{\chd}(0,m) < (1+ \delta) m, \\
&  \mbox{for all $m \ge m_0\}$},
\end{split}
\end{equation}
one has for all $m \ge m_0$ the following inclusion among subsets of $\Gamma(\chu)$ (see below (\ref{1.1}) for notation)
\begin{align}
\{\wt{Z}^\chd_1, \dots, \wt{Z}^\chd_{(1 - \delta)m}\} & \subseteq \{Z^\chd_1, \dots, Z^\chd_{(1 + 3\delta)m}\}, \;\;\mbox{and} \label{4.7}
\\[1ex]
\{Z^\chd_1, \dots, Z^\chd_{(1 - \delta)m}\} & \subseteq \{\wt{Z}^\chd_1, \dots, \wt{Z}^\chd_{(1 + 3\delta)m}\}, \label{4.8}
\end{align}

\medskip\n
where (similarly as above (\ref{2.11})) $\wt{Z}^\chd_v$ and $Z^\chd_v$ stand for $\wt{Z}^\chd_{[v]}$ and $Z^\chd_{[v]}$, when $v \ge 1$, and the sets on the left-hand sides of (\ref{4.7}), (\ref{4.8}) are empty when $(1-\delta) m < 1$. Note that the (favorable) event $\wt{U}^{m_0}_\chd$, which ensures (\ref{4.7}), (\ref{4.8}) is solely defined in terms of the Poisson counting process $(n_\chd (0,t))_{t \ge 0}$, and that as $\chd$ varies over $\chd_z$, $z \in \cC$, these counting processes are i.i.d. (see (\ref{4.4})).

\medskip
We will now extract excursions from $D$ to $\partial U$ and $D'$ to $\partial U'$ from a sequence of excursions from $\chd$ to $\partial \chu$, when $D, \chd$ and $D'$ respectively correspond to $D_z, \chd_z$ and $D_{z'}$, with $z \in \cC$ and $z'$ in $\IL$ neighbor of $z$ (as already mentioned at the beginning of this section, our motivation comes from later having to handle condition (\ref{2.12})). From the infinite sequence of i.i.d. excursions $\wt{Z}^\chd_k, k \ge 1$ (with same law as $X_{\cdot\, \wedge T_{\chu}}$ under $P_{\ov{e}_{\chd}})$, we can extract the successive excursions $\wh{Z}^D_\ell, \ell \ge 1$ and $\wh{Z}^{D'}_\ell, \ell \ge 1$, from $D$ to $\partial U$ and $D'$ to $\partial U'$ that a.s. appear in $\wt{Z}^{\chd}_1,\wt{Z}^{\chd}_2, \dots$ (so for instance $\wh{Z}^D_1, \wh{Z}^D_2, \wh{Z}^D_3$ are the excursions in the order of appearance from $D$ to $\partial U$ in $\wt{Z}^{\check{D}}_1$, then $\wh{Z}^D_4, \wh{Z}^D_5$ are the excursions contained in $\wt{Z}^{\chd}_2$, then $\wt{Z}^{\chd}_3$ contains no such excursion, $\wh{Z}^D_6$ is the unique excursion contained in $\wt{Z}^{\chd}_4$, and so on ...). Note that for a given $\chd$, the sequence $\wh{Z}^D_\ell, \ell \ge 1$, is typically not i.i.d. anymore and the sequences $\wh{Z}^D_\ell, \ell \ge 1$, and $\wh{Z}^{D'}_\ell, \ell \ge 1$, are typically mutually dependent. However, 
\begin{equation}\label{4.9}
\begin{array}{l}
\mbox{under $\IQ^\cC$, the collections $\wh{Z}^{D_z}_\ell, \ell \ge 1$, $\wh{Z}^{D_{z'}}_\ell, \ell \ge 1$, with}
\\
\mbox{$z'$ neighbor of $z$ in $\IL$, are independent, as $z$ varies over $\cC$}
\end{array}
\end{equation}

\medskip\n
(indeed, the $\wt{Z}^{\check{D}_z}_k, k \ge 1$, are independent as $z$ varies over $\cC$, see (\ref{4.4})).

\medskip
We now come to the definition of a favorable event $\wt{G}_B$ that will play an important role in the next section. In the proposition below, we will see that for an adequate choice of parameters the complement of this favorable event has a probability, which decays super-polynomially in $L$.

\medskip
Given $\delta$ and $\kappa$ in $(0,\frac{1}{2})$, an $L$-box $B$ in $\cC$ (i.e. $B = B_z$, with $z \in \cC$), and a neighboring box $B'$ (this precludes that $B'$ is in $\cC$ by (\ref{3.1}), (\ref{3.2})), we set (see (\ref{4.6}) for notation)
\begin{equation}\label{4.10}
\begin{split}
\wt{G}_B = \wt{U}^{m_0}_\chd \cap \Big\{&\wt{Z}^\chd_1,\dots,\wt{Z}^\chd_m \;\mbox{contain at least $(1 - \kappa)\,m$ \mbox{\f $\dis\frac{{\rm cap}(D)}{{\rm cap}(\chd)}$} and at most}
\\
& (1 + \kappa)\,m\, \mbox{\f $\dis\frac{{\rm cap}(D)}{{\rm cap}(\chd)}$} \;\mbox{excursions from $D$ to $\partial U$, for all $m \ge m_0\Big\}$}
\\[1ex]
\underset{B' \,{\rm neighbor}\, B}{\textstyle \bigcap} \Big\{ & \wt{Z}^\chd_1,\dots,\wt{Z}^\chd_m \;\mbox{contain at least $(1 - \kappa)\,m$ \mbox{\f $\dis\frac{{\rm cap}(D)}{{\rm cap}(\chd)}$} and at most}
\\
& (1 + \kappa)\,m\, \mbox{\f $\dis\frac{{\rm cap}(D)}{{\rm cap}(\chd)}$} \;\mbox{excursions from $D'$ to $\partial U'$, for all $m \ge m_0\Big\}$},
\end{split}
\end{equation}

\medskip\n
where $m_0 = [(\log L)^2] + 1$ in the above definition.

\medskip
Here is the main result of this section.

\medskip
\begin{proposition}\label{prop4.1} ($m_0 =[(\log L)^2] + 1$, $\delta, \kappa$ in $(0,\frac{1}{2})$, $\cC$ as in (\ref{3.1}))
\begin{equation}\label{4.11}
\mbox{When $K \ge c_6(\delta)$, then (\ref{4.5}) holds.}
\end{equation}

\medskip\n
When $(1-\delta) m \ge m_0$ and $\delta m_0 \ge 1$, then for $z \in \cC$, $z'$ neighbor of $z$ in $\IL$, $\IQ^\cC$-a.s.~on $\wt{G}_B$,
\begin{equation}\label{4.12}
\left\{\begin{array}{rll}
{\rm i)} &\;\; \{\wh{Z}^D_1, \dots,\wh{Z}^D_\ell\} \subseteq \{Z^D_1,\dots, Z^D_{(1 + \wh{\delta}) \ell}\}, & \mbox{for $\ell \ge m_0$},
\\[1ex]
{\rm ii)} &\;\;  \{Z^D_1, \dots,Z^D_\ell\} \subseteq \{\wh{Z}^D_1,\dots, \wh{Z}^D_{(1 + \wh{\delta}) \ell}\}, & \mbox{for $\ell \ge m_0$},
\end{array}\right.
\end{equation}
\begin{equation}\label{4.13}
\left\{\begin{array}{rll}
{\rm i)} &\;\; \{\wh{Z}^{D'}_1, \dots,\wh{Z}^{D'}_\ell\} \subseteq \{Z^{D'}_1,\dots, Z^{D'}_{(1 + \wh{\delta}) \ell}\}, & \mbox{for $\ell \ge m_0$},
\\[1ex]
{\rm ii)} &\;\;  \{Z^{D'}_1, \dots,Z^{D'}_\ell\} \subseteq \{\wh{Z}^{D'}_1,\dots, \wh{Z}^{D'}_{(1 + \wh{\delta}) \ell}\}, & \mbox{for $\ell \ge m_0$},
\end{array}\right.
\end{equation}
where we have set
\begin{equation}\label{4.14}
1 + \wh{\delta} = \mbox{\f $\dis\frac{1 + \kappa}{1 - \kappa}$} \; \mbox{\f $\dis\frac{(1 + 4 \delta)^2}{(1 - 2 \delta)^2}$}\;.
\end{equation}
Moreover, when $K \ge c_7(\delta, \kappa)$ ($\ge c_6(\delta)$), then
\begin{equation}\label{4.15}
\lim\limits_L \;\mbox{\f $\dis\frac{1}{\log L}$} \; \log \IQ^\cC [\wt{G}_B^c] = - \infty
\end{equation}
(note that the above probability does not depend on the choice of $\cC$, or of $B$ in $\cC$).
\end{proposition}

\begin{proof}
We begin with the proof of (\ref{4.11}). We observe that in the notation of (\ref{4.1}), (\ref{4.2}), for $x \in \partial \check{V}$, $y \in \chd$
\begin{align*}
Q_x[Y = y] & \stackrel{(\ref{4.2})}{=} P_x[X_{H_\chc} = y, H_\chc < \infty] + P_x[H_\chc = \infty] \,\ov{e}_\chc(y)
\\
&\;\, = \;P_x[X_{H_\chc} = y | H_\chc = H_\chd < \infty] \, P_x[H_\chc = H_\chd < \infty] +\ov{e}_\chc(y)\,P_x[H_\chc = \infty].
\end{align*}
As a result we see that
\begin{equation}\label{4.16}
\begin{split}
Q_x [Y = y|Y \in \chd]  = &\; \big(P_x[X_{H_\chc} = y|H_\chc = H_\chd < \infty]\,P_x[H_\chc = H_\chd < \infty]  
\\
&+ \; \mbox{\f $\dis\frac{e_\chc(y)}{e_{\chc}(\chd)} \,P_x[H_{\chc} = \infty] \,\ov{e}_\chc(\chd)\big)$}
\\[1ex]
\times &\;  \big(\dsl_{y' \in \check{D}} \mbox{the same terms with $y'$ in place of $y\big)^{-1}$}.
\end{split}
\end{equation}

\n
We can apply Proposition \ref{prop1.5} with the choice $A = \chd - z$ and $B = \chc - z$, if $\chd = \chd_z$ (note in particular that by (\ref{3.1}), (\ref{3.2}), $\IZ^d \backslash (\chc - z)$ is connected). When $K \ge c(\delta)$, we thus find that for all $x \in \partial \chv$, $y' \in \chd$, $P_x[H_{H_\chc} = y'| H_\chc = H_\chd < \infty]$ and $\frac{e_\chc(y')}{e_\chc(\chd)}$ lie between $(1 - \frac{\delta}{10}) \,\ov{e}_\chd(y')$ and $(1 + \frac{\delta}{10}) \,\ov{e}_\chd(y')$, see (\ref{1.25}), (\ref{1.26}). Inserting these bounds inside (\ref{4.16}), and using that $1 - \frac{\delta}{3} \le \frac{1 - \delta/10}{1 + \delta/10}$, as well as $1 + \frac{\delta}{3} \ge \frac{1 + \delta/10}{1 - \delta/10}$, we see that (\ref{4.5}) holds. This proves (\ref{4.11}).

\medskip
We now turn to the proof of (\ref{4.12}). We assume that $K \ge c_6(\delta)$ so that (\ref{4.5}) holds. Hence, we have (\ref{4.7}), (\ref{4.8}) on $\wt{G}_B \subseteq \wt{U}^{m_0}_\chd$ (recall that presently $m_0 = [(\log L)^2] + 1$). We thus find that when $m \ge m_0$
\begin{equation}\label{4.17}
\{\wt{Z}_1^\chd,\dots,\wt{Z}^\chd_{(1 - \delta)m}\} \subseteq \{Z^\chd_1,\dots,Z^\chd_{(1 + 3 \delta)m}\} \subseteq \big\{\wt{Z}^\chd_1,\dots,\wt{Z}^\chd_{\frac{(1 + 4 \delta)^2}{(1 - \delta)}m} \,\big\},
\end{equation}
where, for the last inclusion, we have used that $m' = [\frac{[(1 + 3 \delta) m]}{1 - \delta}] + 1$ ($\ge m_0$) satisfies $(1 - \delta) m' \ge [(1 + 3 \delta)m]$ and $m' \le \frac{(1 + 4 \delta)}{1-\delta}m$, since $m \delta \ge m_0 \,\delta \ge 1$, so that $( 1 + 3 \delta) \,m' \le \frac{(1+ 4 \delta)^2}{1 - \delta} \,m$.

\medskip
We will first prove (\ref{4.12}) i). By the definition (\ref{4.10}), when $(1 - \delta)m \ge m_0$, and $\delta m_0 \ge 1$, on $\wt{G}_B$ the set of excursions on the left-hand side of (\ref{4.17}) contains at least $(1 - \kappa)[(1-\delta)m] \,\frac{{\rm cap}(D)}{{\rm cap}(\chd)} \ge (1-\kappa)(1 - 2 \delta) m \;\frac{{\rm cap}(D)}{{\rm cap}(\chd)} \stackrel{\rm def}{=} t$ excursions from $D$ to $\partial U$, and the set of excursions on the right-hand side of (\ref{4.17}) contains no more than $[(1 + \kappa) \,\frac{(1 + 4 \delta)^2}{1 - \delta} \,m\frac{{\rm cap}(D)}{{\rm cap}(\chd)}] \le (1+\wt{\delta})\,t$  excursions from $D$ to $\partial U$, where $1 +\wt{\delta} = \frac{1 + \kappa}{1 - \kappa} \;\frac{(1 + 4 \delta)^2}{(1- \delta)(1 - 2 \delta)}$.

\medskip
Hence, looking at the ``first $t$ excursions'' $\wh{Z}^D_\ell$, $1 \le \ell \le t$ (which are inscribed in the set of excursions on the left-hand side of (\ref{4.17})) and at the ``first $(1 + \wt{\delta})\,t$ excursions'' $Z^D_\ell$, $1 \le \ell \le (1 + \wt{\delta}) \,t$ (which exhaust all excursions from $D$ to $\partial U$ inscribed within the set of excursions in the middle of (\ref{4.17})), we see that on $\wt{G}_B$, when $(1- \delta)  m \ge m_0$ and $\delta m_0 \ge 1$,
\begin{equation}\label{4.18}
\{\wh{Z}^D_1, \dots,\wh{Z}^D_t\} \subseteq \{Z^D_1,\dots, Z^D_{(1 + \wt{\delta})t}\}.
\end{equation}

\n
As $m$ varies over the range $(1-\delta) m \ge m_0$, note  that $[t]$ covers in particular all integers with value bigger or equal to $m_0$, and when $[t] \ge m_0$, then, since $\delta[t] \ge \delta m_0 \ge 1$, 
\begin{equation*}
(1 + \wt{\delta}) \,t \le (1 + \wt{\delta})([t] + \delta [t]) = (1 + \wt{\delta})(1 + \delta)[t] \le (1 + \wh{\delta}) [t],
\end{equation*}
where we used that $1 + \delta \le (1- \delta)/(1 - 2 \delta)$ in the last step. Thus (\ref{4.12}) i) follows from (\ref{4.18}). Similarly, for (\ref{4.12}) ii), we can look at the ``first $t$ excursions'' $Z^D_\ell$, $1 \le \ell \le t$, from $D$ to $\partial U$ (which are inscribed in the set of excursions in the middle of (\ref{4.17})), and at the ``first $(1 + \wt{\delta})\,t$'' excursions $\wh{Z}^D_\ell$, $1 \le \ell \le (1 + \wt{\delta})\,t$, from $D$ to $\partial U$ (which exhaust all excursions from $D$ to $\partial U$ inscribed within the set of excursions on the right-hand side of (\ref{4.17})), to infer from (\ref{4.17}) that on $\wt{G}_B$, when $(1 - \delta) \,m \ge m_0$ and $m_0 \delta \ge 1$, we have
\begin{equation}\label{4.19}
\{Z^D_1,\dots,Z^D_t\} \subseteq \{\wh{Z}^D_1,\dots, \wh{Z}^D_{(1 + \wt{\delta})t}\},
\end{equation}
and conclude in the same fashion that (\ref{4.12}) ii) holds.

\medskip
Of course by the same reasoning as above, with $D'$ in place of $D$, we obtain (\ref{4.13}) i) and ii).

\medskip
We now turn to the proof of (\ref{4.15}). It is plain from the definition of $\wt{G}_B$ in (\ref{4.10}) and from (\ref{4.4}) that the probability in (\ref{4.15}) does not depend on the choice of $\cC$ satisfying (\ref{3.1}) and of $B$ in $\cC$. We will first bound $\IQ^\cC((\wt{U}_\chd^{m_0})^c)$, see (\ref{4.6}) (with $m_0 = [(\log L)^2] + 1$). Note that $n_{\chd}(a,b) = n_{\chd}(0,b) - n_\chd(0,a)$, for $0 \le a \le b$, and since $n_\chd(0,t), t \ge 0$, is a Poisson counting function of unit intensity (see (\ref{4.4})), it follows from a standard exponential Chebyshev estimate that
\begin{equation}\label{4.20}
\overline{\lim\limits_L} \;\mbox{\f $\dis\frac{1}{m_0}$} \;\log \IQ^\cC\big((U^{m_0}_\chd)^c\big) \le - c(\delta) < 0.
\end{equation}
We now control the $\IQ^\cC$ probability of the complement of the event that appears after the intersection on the first line of (\ref{4.10}). We will use the following simple fact:

\begin{lemma}\label{lem4.2} $(L \ge 1, K \ge c)$
\begin{equation}\label{4.21}
\mbox{\f $\dis\frac{{\rm cap}(D)}{{\rm cap}(\chd)}$} \ge P_{\ov{e}_\chd} [H_D < T_\chu] \ge \mbox{\f $\dis\frac{{\rm cap}(D)}{{\rm cap}(\chd)}$} - \mbox{\f $\dis\frac{c}{K^{d-2}}$} \ge \mbox{\f $\dis\frac{{\rm cap}(D)}{{\rm cap}(\chd)}$}\;\Big(1 -  \mbox{\f $\dis\frac{c'}{K^{d-2}}$}\Big).
\end{equation}
\end{lemma}

\begin{proof}
By the sweeping identity (\ref{1.11}), we have
\begin{equation}\label{4.22}
{\rm cap}(D) = P_{e_\chd}[H_D < \infty] = P_{e_{\chd}}[H_D < T_\chu] + P_{e_\chd}[T_{\chu} < H_D < \infty].
\end{equation}

\n
The first inequality of (\ref{4.21}) follows by dividing the above equalities by ${\rm cap}(\chd)$. To obtain the last two inequalities of (\ref{4.21}), we observe that it follows from (\ref{4.22}) that
\begin{align*}
P_{e_{\chd}}[H_D < T_{\chu}] = & \;{\rm cap}(D) - P_{e_{\chd}} [T_{\chu} < H_D < \infty]
\\[1ex]
\ge &\; {\rm cap}(D) - {\rm cap}(\chd)\;\sup\limits_{\partial \check{U}} P_x[H_D < \infty]
\\[0.5ex]
\ge &\; {\rm cap}(D) - {\rm cap}(\chd) \;\mbox{\f $\dis\frac{c}{K^{d-2}}$} \;\; \mbox{(by (\ref{1.10}), (\ref{1.8}), (\ref{2.9}))}
\\[1ex]
\ge &\; {\rm cap}(D) \Big(1 - \mbox{\f $\dis\frac{c'}{K^{d-2}}$}\Big) \;\; \mbox{(by (\ref{1.8}) and (\ref{2.9}))}.
\end{align*}
Dividing by ${\rm cap}(\chd)$ we obtain the last two inequalities of (\ref{4.21}).
\end{proof}

\medskip
We now resume the bound on the $\IQ^\cC$-probability of the complement of the event after the intersection on the first line of (\ref{4.10}). We first control $\IQ^\ell (\cup_{m \ge m_0} F_m)$, where 
\begin{equation*}
\mbox{$F_m = \{\wt{Z}^{\chd}_1,\dots,\wt{Z}^\chd_m$ contains fewer than $(1 - \kappa)\;\mbox{\f $\dis \frac{{\rm cap}(D)}{{\rm cap}(\chd)}$} \,m$ excursions from $D$ to $\partial U\}$}.
\end{equation*}

\n
We note that the $\wt{Z}^\chd_k$, $k \ge 1$, are i.i.d. and each contain at least one excursion from $D$ to $\partial U$ with probability $p = P_{\ov{e}_\chd} [H_D < T_{\chu}]$. Thus, by (\ref{4.21}), we can assume that for $K \ge c(\kappa)$,
\begin{equation}\label{4.23}
p \ge \Big(1 - \mbox{\f $\dis\frac{\kappa}{2}$}\Big) \; \mbox{\f $\dis \frac{{\rm cap}(D)}{{\rm cap}(\chd)}$} \;.
\end{equation}
Hence, for any $m \ge 1$, with $Y_i, i \ge 1$, i.i.d. Bernoulli variables with success probability $p$,
\begin{equation}\label{4.24}
\begin{split}
\IQ^\cC[F_m] & \le P\Big[\dsl^m_{i=1} Y_i \le (1- \kappa) \; \mbox{\f $\dis \frac{{\rm cap}(D)}{{\rm cap}(\chd)}$} \;m\Big]
\\
& \le \;\exp\{- m \,I(\wt{p})\}, \;\;\mbox{by standard exponential Chebyshev bounds},
\end{split}
\end{equation}
where we have set
\begin{equation}\label{4.25}
\left\{ \begin{array}{l}
\wt{p} = (1 - \kappa) \; \mbox{\f $\dis \frac{{\rm cap}(D)}{{\rm cap}(\chd)}$} \; (< p \;\mbox{by (\ref{4.23})) and}
\\[2ex]
I(a) = a \,\log \Big(\mbox{\f $\dis\frac{a}{p}$}\Big) + (1 - a) \,\log \,\Big(\mbox{\f $\dis\frac{1-a}{1-p}$}\Big), \; \mbox{for $0 \le a \le 1$}.
\end{array}\right.
\end{equation}
Note that
\begin{equation}\label{4.26}
\begin{split}
I(p) = 0 = & \;\mbox{$I'(p)$ and $I''(a) = \mbox{\f $\dis\frac{1}{a}$} + \mbox{\f $\dis\frac{1}{1-a}$}$, for $0 < a < 1$, so that}
\\[1ex]
I(\wt{p})  = & \dil^p_{\wt{p}} - I'(r)dr = \dil^p_{\wt{p}} dr \,\dil^p_r \mbox{\f $\dis\frac{1}{t}$} + \mbox{\f $\dis\frac{1}{1-t}$} \;dt = \dil_{\wt{p} < r < t < p} \;\mbox{\f $\dis\frac{1}{t}$} + \mbox{\f $\dis\frac{1}{1-t}$} \;dt\,dr
\\[1ex]
 \ge& \;c(p-\wt{p})^2 \ge c'\,\kappa^2 .
\end{split}
\end{equation}
It thus follows that when $K \ge c(\kappa)$, then
\begin{equation}\label{4.27}
\IQ^\ell \big[\textstyle \bigcup\limits_{m \ge m_0} F_m\big] \le \dsl_{m \ge m_0} e^{-c' \kappa^2 m} = (1 - e^{-c' \kappa^2})^{-1} \,e^{-c' \kappa^2 m_0} .
\end{equation}

\n
We now continue to bound the $\IQ^\cC$-probability of the complement of the event after the intersection on the first line of (\ref{4.1}). We will now bound $\IQ^\cC[\bigcup_{m \ge m_0} H_m]$, where
\begin{equation*}
\mbox{$H_m = \Big\{\wt{Z}^\chd_1,\dots,\wt{Z}^\chd_m$ contain more than $(1 + \kappa) \;\mbox{\f $\dis \frac{{\rm cap}(D)}{{\rm cap}(\chd)}$} \;m$ excursions from $D$ to $\partial U\Big\}$.}
\end{equation*}

\n
Observe that any $\wt{Z}^\chd_\ell$ is distributed as the continuous-time simple random walk with starting distribution $\ov{e}_\chd$, stopped when exiting $\chu$. Hence, the number of excursions from $D$ to $\partial U$ that $\wt{Z}^\chd_\ell$ contains is stochastically dominated by the number $N_{D,U}$ of excursions from $D$ to $\partial U$, see (\ref{1.30}), for the simple random walk with starting distribution $\ov{e}_\chd$. Moreover, setting $p = \inf_{x \in \partial U} P_x [H_D = \infty] \ge 1 - \frac{c}{K^{d-2}}$ (see (\ref{1.18})), we find that for $\lambda > 0$ such that $e^\lambda( 1- p) < 1$,
\begin{equation}\label{4.28}
\begin{split}
E_{\ov{e}_\chd} [e^{\lambda N_{D,U}}] & = P_{\ov{e}_\chd} [H_D = \infty] + E_{\ov{e}_\chd} \big[H_D < \infty, E_{X_{H_D}} [e^{\lambda N_{D,U}}]\big]
\\
& = \; 1 + E_{\ov{e}_\chd}\big[H_D < \infty, E_{X_{H_D}}[e^{\lambda N_{D,U}}]-1\big]
\\
&\!\!\! \stackrel{(\ref{1.31})}{\le} 1 + P_{\ov{e}_\chd} [H_D < \infty] \Big(\mbox{\f $\dis\frac{e^\lambda p}{1 - e^\lambda(1-p)}$} - 1\Big)
\\ 
&\! \stackrel{}{=}\;  1 + \mbox{\f $\dis \frac{{\rm cap}(D)}{{\rm cap}(\chd)}$} \; \mbox{\f $\dis\frac{e^\lambda -1}{1 - e^\lambda(1-p)}$} , \;\;\mbox{with $p \ge 1 - \mbox{\f $\dis\frac{c}{K^{d-2}}$}$} \,.
\end{split}
\end{equation}

\n
Then, by (\ref{4.28}) and an exponential Chebyshev estimate, we see that for $K \ge \wt{c}$, $0 < \lambda < c' \log K$ (so that $e^\lambda (1-p) < 1$), and $m \ge 1$,
\begin{equation}\label{4.29}
\begin{split}
\IQ^\cC[H_m] & \le \exp\Big\{- m\Big[\lambda (1 + \kappa) \;\mbox{\f $\dis \frac{{\rm cap}(D)}{{\rm cap}(\chd)}$} - \log \Big( 1 + \mbox{\f $\dis \frac{{\rm cap}(D)}{{\rm cap}(\chd)}$} \;\mbox{\f $\dis\frac{e^\lambda -1}{1 - e^\lambda(1-p)}$}\Big)\Big]\Big\}
\\
&\!\!\!\!\!\!\!\! \stackrel{\log(1+a) \le a}{\le} \exp\Big\{- m \; \mbox{\f $\dis \frac{{\rm cap}(D)}{{\rm cap}(\chd)}$}  \; \Big[\lambda (1 + \kappa) - \mbox{\f $\dis \frac{e^\lambda -1}{1 - e^\lambda (1-p)}$}\Big]\Big\}. 
\end{split}
\end{equation}

\n
If $K \ge c(\kappa)$ we can pick $\lambda = \log (1 + \frac{\kappa}{2})$ and assume that $1 - (1 + \frac{\kappa}{2})$ $\frac{c}{K^{d-2}} \ge \frac{1 + \kappa/2}{1 + \kappa}$ with $c$ as in the last line of (\ref{4.28}). It now follows that
\begin{equation*}
\begin{array}{l}
\lambda(1 + \kappa) - \mbox{\f $\dis\frac{e^\lambda -1}{1-e^\lambda(1-p)}$} \ge (1 + \kappa) \Big[\log \Big(1 + \mbox{\f $\dis\frac{\kappa}{2}$}\Big) - \mbox{\f $\dis\frac{\kappa/2}{1 + \kappa/2}$}\Big] = 
\\[2ex]
(1 + \kappa) \dil^{\kappa/2}_0 \; \mbox{\f $\dis\frac{1}{1+u}$} - \mbox{\f $\dis\frac{1}{1+\kappa/2}$}\;du \stackrel{\rm def}{=} \psi(\kappa) > 0.
\end{array}
\end{equation*}

\n
Since ${\rm cap}(D)/{\rm cap}(\chd) \ge c$, coming back to (\ref{4.29}), we see that for $K \ge c(\kappa)$,
\begin{equation}\label{4.30}
\IQ^\cC \big[\textstyle \bigcup\limits_{m \ge m_0} H_m\big] \le \dsl_{ m \ge m_0} e^{-c \psi(\kappa)m} = (1 - e^{-c \psi(\kappa)})^{-1} \,e^{-c \psi(\kappa)m_0} .
\end{equation}

\n
Now (\ref{4.27}) and (\ref{4.30}) are more than enough to show that when $K \ge c(\kappa)$, the $\IQ^\cC$-probability of the complement of the event after the intersection on the first line of (\ref{4.10}) decays super-polynomially in $L$ (recall $m_0 = [(\log L)^2] + 1$). The $\IQ^\cC$-probability of the complement of the event on the last two lines of (\ref{4.10}) (relative to excursions from $D'$ to $\partial U')$ is handled in the same fashion. Keeping in mind (\ref{4.20}), this completes the proof of (\ref{4.15}), and hence of Proposition \ref{prop4.1}.
\end{proof}

\begin{remark}\label{rem4.3} \rm 
The interest of the ``favorable'' events $\wt{G}_B$ is threefold. By construction these events are independent under $\IQ^\cC$ as $B$ varies over $\cC$ (see (\ref{4.4})). When $K\ge c_7(\delta, \kappa)$, for large $L$, the events $\wt{G}_B$ are very likely (see (\ref{4.15})). And finally, on $\wt{G}_B$ we can control the excursions from $D$ to $\partial U$ and $D'$ to $\partial U'$ in the random interlacements, via (\ref{4.12}), (\ref{4.13}). This feature will be very handy in the next section. \hfill $\square$
\end{remark}

\section{Super-exponential decay}
\setcounter{equation}{0}

In this section we relate the scale $L$ governing the boxes introduced in Section 2 to the scale $N$ of basic interest that appears in the main statement (\ref{0.6}). Theorem \ref{theo6.3} of the next section will prove this main statement. The present section contains an important preparation. We introduce certain columns of $L$-boxes going from $B_N$ to $S_N$ (see (\ref{0.3}) for notation), and show in Theorem \ref{theo5.1} that given $\alpha > \beta > \gamma$ smaller than the critical value $\ov{u}$ (see (\ref{2.3})), except on an event of super-exponentially decaying probability at rate $N^{d-2}$, only a small fraction of columns contain a bad$(\alpha, \beta,\gamma)$ box (see (\ref{2.11})~-~(\ref{2.13})). This super-exponential estimate relies on the soft local time couplings constructed in Section 4. Theorem \ref{theo5.1} roughly plays the role of Proposition 5.4 of \cite{Szni} in the context of level-set percolation for the Gaussian free field. However, whereas the various independence properties used in the proof of Proposition 5.4 of \cite{Szni} were easy facts pertaining to the built-in independence of the local fields of \cite{Szni}, here, our main tools will come from Proposition \ref{prop4.1} and Remark \ref{rem4.3}.

\medskip
We now relate the parameter $L$ to the parameter $N$ entering (\ref{0.3}) (and the main upper-bound (\ref{0.6}) we are aiming at). To this end, we introduce a large $\Gamma \ge 1$ and set
\begin{equation}\label{5.1}
L = [(\Gamma N \log N)^{\frac{1}{d-1}}] \quad \mbox{(we recall that $\IL = L \IZ^d$, see (\ref{2.6}))}.
\end{equation}
We also consider $K \ge 100$.

\medskip
We still need to introduce additional notation and definitions, in particular, regarding columns. Given $e \in \IZ^d$, with $|e| = 1$, and $N > 1$, we denote by $F_{e,N}$ the face in the direction $e$ of $B_N$, that is $F_{e,N} = \{x \in B_N; x \cdot e = N\}$. For each face we consider the set of columns (attached to the face), where a column consists of $L$-boxes $B$ contained in $\{z \in \IZ^d; x \cdot e > N\} \cap B_{(M+1)N}$, with same projection in the $e$-direction on $\{x \in \IZ^d; x \cdot e = N\}$, which we require to be contained in the face $F_{e,N}$ of $B_N$ (we recall that $M > 1$ is an arbitrary number, that $S_N = \{x \in \IZ^d; |x|_\infty = [MN]\}$, see (\ref{0.3}), and that to be an $L$-box means that $B = B_z$, with $z \in \IL$, see (\ref{2.10})). Incidentally, note that for large $N$ the total number of columns is bounded by $c(\frac{N}{L})^{d-1} \le \frac{c'}{\Gamma} \;\frac{N^{d-2}}{\log N}$.

\medskip
We now consider $\ov{u}$ from (\ref{2.3}) as well as
\begin{equation}\label{5.2}
\alpha > \beta > \gamma \; {\rm in} \; (0,\ov{u}).
\end{equation}
    
\medskip\n
We view $\delta, \kappa$ in $(0,\frac{1}{2})$ as functions of $\alpha, \beta, \gamma$ such that       
\begin{equation}\label{5.3}
\ov{u} > ( 1 + \wh{\delta})^2 \alpha, \; \mbox{\f $\dis\frac{\alpha}{(1 + \wh{\delta})^2}$} > (1 + \wh{\delta})^2\beta, \; \mbox{and}  \; \mbox{\f $\dis\frac{\beta}{(1 + \wh{\delta})^2}$} > \gamma,
\end{equation}  

\n
with $\wh{\delta}$ as (\ref{4.14}). We now define (see (\ref{2.11}) - (\ref{2.13}) and (\ref{4.10}) for notation)      
\begin{equation}\label{5.4}
\begin{split}
\eta = & \;\mbox{$\IP\big[B$ is bad$\big((1 + \wh{\delta})^2 \alpha, \beta,\gamma\big)\big] + \IP\Big[B$ is bad$\Big(\mbox{\f $\dis\frac{\alpha}{(1 + \wh{\delta})^2}$}, (1 +  \wh{\delta})^2 \beta, \alpha\Big)\Big] \; + $}
\\[-1ex]
& \; \mbox{$\IP[B$ is bad$\Big(\alpha ,\, \mbox{\f $\dis\frac{\beta}{(1 + \wh{\delta})^2}$}, \gamma\Big)\Big] + 3 \IQ^\cC[\wt{G}^c_B]$}
\end{split}
\end{equation}
\n
(the last probability does not depend on $\cC$ satisfying (\ref{3.1}), nor on $B$ in $\cC$, see below (\ref{4.15})).

\medskip
We assume from now on that $K \ge c_8(\alpha, \beta,\gamma)$, so that Theorem \ref{theo2.3} applies to each of the first three probabilities in the right-hand side of (\ref{5.4}) and (\ref{4.15}) of Proposition \ref{prop4.1} applies (we recall that we view $\delta$ and $\kappa$ as functions of $\alpha,\beta,\gamma$). We thus see that
\begin{equation}\label{5.5}
\lim\limits_{L \r \infty} \;\mbox{\f $\dis\frac{1}{\log L}$} \; \log \eta = - \infty,
\end{equation}
i.e. $\eta$ tends to zero super-polynomially in $L$. We can then define
\begin{equation}\label{5.6}
\rho = \sqrt{\mbox{\f $\dis\frac{\log L}{\log \frac{1}{\eta}}$}} \underset{L \r \infty}{\longrightarrow} 0.
\end{equation}
 
\medskip \n
We will see that the following event is ``negligible'' for our purpose. Namely, we set
\begin{equation}\label{5.7}
\mbox{$C_N = \big\{$there are at least $\rho\big(\frac{N}{L}\big)^{d-1}$ columns containing a bad$(\alpha,\beta,\gamma)$ box$\big\}$}.
\end{equation} 
Here is the main result of this section. Recall that by (\ref{5.2}) $\alpha > \beta > \gamma $ are in $(0,\ov{u})$.

\begin{theorem}\label{theo5.1} (super-exponential bound)

\medskip
When $K \ge c_8 (\alpha,\beta,\gamma)$,
\begin{equation}\label{5.8}
\lim\limits_N \; \mbox{\f $\dis\frac{1}{N^{d-2}}$} \; \log \IP [C_N] = - \infty .
\end{equation} 
\end{theorem}

\begin{proof}
We write $\IL ( = L \IZ^d)$ as the disjoint union of the sets $y + \ov{K} \,\IL$, where $y$ varies over $\{0,L,2L,\dots, (\ov{K} - 1)L\}^d$ and $\ov{K} = 2 K + 3$ (see (\ref{3.1})). We then find that (in the sum below,   $y$ ranges over the set just mentioned)  
\begin{equation*}
\begin{split}
\IP[C_N] \le  \dsl_y\; \IP\Big[ & \mbox{at least $\mbox{\f $\dis\frac{\rho}{\ov{K}^d}$}\; \Big(\mbox{\f $\dis\frac{N}{L}$}\Big)^{d-1}$ boxes $B_z$, with $B_z$ in some column, and}
\\[-0.5ex]
&\mbox{$z \in y + \ov{K}\IL$ are bad$(\alpha,\beta,\gamma)\Big]$}.
\end{split}
\end{equation*}

\n
Given any $y \in \{0,L, 2L , \dots, (\ov{K} - 1)L\}^d$, we denote by $\cC_y$ the collection of $L$-boxes $B_z$, with $z \in y + \ov{K} \IL$ that are contained in some column. When $N$ is large, $\cC_y$ is non-empty and satisfies (\ref{3.1}), and we consider the probability $\IQ^{\cC_y}$ (see below (\ref{4.2})). We find that in the notation of (\ref{4.10})
\begin{equation}\label{5.9}
\begin{split}
\IP[C_N]  \le &\; \dsl_y \IQ^{\cC_y}\big[\mbox{at least $\mbox{\f $\dis\frac{\rho}{2 \ov{K}^d}$} \big(\mbox{\f $\dis\frac{N}{L}$}\big)^{d-1}$ boxes $B$ in $\cC_y$ are such that $\wt{G}^c_B$ holds$\big]$} \; +
\\
& \;\dsl_y \IQ^{\cC_y}\big[\mbox{at least $\mbox{\f $\dis\frac{\rho}{2 \ov{K}^d}$} \big(\mbox{\f $\dis\frac{N}{L}$}\big)^{d-1}$ boxes $B$ in $\cC_y$ are bad$(\alpha, \beta, \gamma)$ and}
\\[-0.5ex] 
&\qquad \qquad  \mbox{such that $\wt{G}_B$ occurs$\big] \stackrel{\rm def}{=} A_1 + A_2$}.
\end{split}
\end{equation}

\medskip\n
We first bound $A_1$. We note that the events $\wt{G}^c_B$ are independent under $\IQ^{\cC_y}$, as $B$ varies over $\cC_y$, see Remark \ref{rem4.3}, and identically distributed (see below (\ref{4.15})), so that
\begin{equation}\label{5.10}
A_1 \le \ov{K}^d \, P\Big[\dsl^m_1 Y_i \ge \mbox{\f $\dis\frac{\rho}{2 \ov{K}^d}$} \;\big(\mbox{\f $\dis\frac{N}{L}$}\big)^{d-1}\Big],
\end{equation}

\n
where $Y_i$, $i \ge 1$, are i.i.d. Bernoulli variables with success probability $\eta \ge \IQ^{\cC_y}[\wt{G}^c_B]$ (by the choice of $\eta$ in (\ref{5.4})), and $m$ denote the total number of boxes in all columns), so that for large $N$,
\begin{equation}\label{5.11}
c\Big(\mbox{\f $\dis\frac{N}{L}$}\Big)^d \le m \le c' \Big(\mbox{\f $\dis\frac{MN}{L}$}\Big)^d.
\end{equation}
Then, using the super-polynomial decay in $L$ of $\eta$ in (\ref{5.5}), the same calculation as between (5.22) and (5.27) in Proposition 5.4 of \cite{Szni} shows that
\begin{equation}\label{5.12}
\lim\limits_N \; \mbox{\f $\dis\frac{1}{N^{d-2}}$} \;\log A_1 = - \infty.
\end{equation}

\n
We then turn to the control of $A_2$. We observe that when $B \in \cC_y$ is such that $\wt{G}_B$ holds and $B$ is bad$(\alpha, \beta, \gamma)$, then some of (\ref{2.11}), (\ref{2.12}), (\ref{2.13}) fail for $B$. So, for large $N$, either

\bigskip\n
a) $B \backslash ({\rm range} \,Z_1^D \cup \ldots \cup {\rm range} \,Z^D_{\alpha\,{\rm cap}(D)}$) does not contain a connected set with diameter at least $\frac{L}{10}$, hence by (\ref{4.12}) ii)
\begin{equation}\label{5.13}
\begin{array}{l}
\mbox{$B \backslash ({\rm range} \,\wh{Z}_1^D \cup \ldots \cup {\rm range} \,\wh{Z}^D_{\alpha(1 + \wh{\delta}){\rm cap}(D)}$) does not contain a connected set}
\\
\mbox{of diameter at least $\frac{L}{10}$},
\end{array}
\end{equation}
or

\medskip\n
b) there are connected sets with diameter at least $\frac{L}{10}$ in $B \backslash ({\rm range} \,Z_1^D \cup \ldots \cup {\rm range} \,Z^D_{\alpha\,{\rm cap}(D)}$) and $B' \backslash ({\rm range} \,Z_1^{D'} \cup \ldots \cup {\rm range} \,Z^{D'}_{\alpha\,{\rm cap}(D')}$), with $B'$ some neighboring box of $B$, which are not linked by a path in $D \backslash ({\rm range} \,Z_1^D \cup \ldots \cup {\rm range} \,Z^D_{\beta\,{\rm cap}(D)}$). Hence, by (\ref{4.12}) i) and (\ref{4.13}) i) (with $\ell = [\frac{\alpha}{(1 + \wh{\delta})} \,{\rm cap}(D)]$) and (\ref{4.12}) ii) (with $\ell = [\beta \,{\rm cap}(D)]$) we find that
\begin{equation}\label{5.14}
\begin{array}{l}
\mbox{there are connected sets of diameter at least $\mbox{\f $\dis\frac{L}{10}$}$ in} 
\\
\mbox{$B \backslash ({\rm range} \,\wh{Z}_1^D \cup \ldots \cup {\rm range} \,\wh{Z}^D_{\frac{\alpha}{1 + \wh{\delta}} \, {\rm cap}(D)}$) and}
\\[1.3ex]
\mbox{$B' \backslash ({\rm range} \,\wh{Z}_1^{D'} \cup \ldots \cup {\rm range} \,\wh{Z}^{D'}_{\frac{\alpha}{1 + \wh{\delta}} \, {\rm cap}(D)}$), with}
\\[1.3ex]
\mbox{$B'$ some neighboring box of $B$, which are not connected}
\\[0.8ex]
\mbox{by a path in $D \backslash ({\rm range} \,\wh{Z}_1^D \cup \ldots \cup {\rm range} \,\wh{Z}^D_{(1 + \wh{\delta}) \beta\, {\rm cap}(D)}$),}
\end{array}
\end{equation}
or

\medskip\n
c) we have
\begin{equation*}
\dsl_{1 \le \ell \le \beta\,{\rm cap}(D)} \; \dil^{T_U}_0 e_D \big(Z^D_\ell (s)\big)ds < \gamma \,{\rm cap}(D)
\end{equation*}

\medskip\n
so that by (\ref{4.12}) i) (with $\ell = [\frac{\beta}{(1 + \wh{\delta})} \,{\rm cap}(D)])$, we have
\begin{equation}\label{5.15}
\dsl_{1 \le \ell \le \frac{\beta}{1 + \wh{\delta}} \,{\rm cap}(D)} \; \dil^{T_U}_0 e_D \big(\wh{Z}^D_\ell(s)\big)ds < \gamma \,{\rm cap}(D).
\end{equation}

\n
As a result, we have obtained that for large $N$, and each $y$ in $\{0,L,2L,\dots,(\ov{K}-1)L\}^d$ the probability in the sum defining $A_2$ in (\ref{5.9}) satisfies
\begin{equation}\label{5.16}
\begin{split}
\IQ^{\cC_y} [&\mbox{at least $\frac{\rho}{2\ov{K}^d} \,(\frac{N}{L})^{d-1}$ boxes $B$ in $\cC_y$ are bad$(\alpha, \beta, \gamma)$ and such that}
\\ 
&\mbox{$\wt{G}_B$ occurs$]$} \le
\\[1ex]
\IQ^{\cC_y} [&\mbox{there are at least $\frac{\rho}{2\ov{K}^d} \,(\frac{N}{L})^{d-1}$ boxes $B$ in $\cC_y$ for which $\wh{a}_B \cup \wh{b}_B \cup \wh{c}_B$ holds$]$},
\end{split}
\end{equation}

\n
where $\wh{a}_B$ refers to the event in (\ref{5.13}), $\wh{b}_B$ refers to the event in (\ref{5.14}), and $\wh{c}_B$ to the event in (\ref{5.15}). By (\ref{4.9}), as $B$ varies over $\cC_y$, the events $\wh{a}_B \cup \wh{b}_B \cup \wh{c}_B$ are independent under $\IQ^{\cC_y}$.

\medskip
In addition, for large $N$, for any $y \in \{0,L, 2L, \dots, (\ov{K}-1)L\}^d$, and $B$ in $\cC_y$
\begin{equation}\label{5.17}
\begin{array}{l}
\IQ^{\cC_y} [\wh{a}_B] \le \IQ^{\cC_y} [\wt{G}^c_B] + \IQ^{\cC_y} [\wh{a}_B \cap \wt{G}_B] \le \IQ^{\cC_y} [\wt{G}^c_B] \; +
\\[1ex]
\IP[B \backslash({\rm range} \,Z^D_1 \cup \ldots \cup Z^D_{(1 + \wh{\delta})^2 \alpha \,{\rm cap}(D)}) \;\mbox{does not contain a connected set}
\\
\quad \mbox{of diameter $\ge \frac{L}{10}$]},
\end{array}
\end{equation}

\n
using (\ref{4.12}) i) in the last inequality (with $\ell = [(1 + \wh{\delta})\,{\rm cap}(D)]$).

\medskip
Similarly, by (\ref{4.12}) ii), (\ref{4.13}) ii) (with $\ell = [\frac{\alpha}{(1 + \wh{\delta})^2} \, {\rm cap}(D)]$) and (\ref{4.12}) i) (with $\ell = [(1 + \wh{\delta}) \beta \, {\rm cap}(D)]$), we find that
\begin{equation}\label{5.18}
\begin{array}{l}
\IQ^{\cC_y} [\wh{b}_B] \le \IQ^{\cC_y} [\wt{G}^c_B] + \IQ^{\cC_y} [\wh{b}_B \cap \wt{G}_B] \stackrel{(\ref{4.12}),(\ref{4.13})}{\le} \IQ^{\cC_y} [\wt{G}^c_B] \; +
\\[1ex]
\IP\big[\mbox{there are connected sets of diameter at least $\frac{L}{10}$ in}
\\[1ex]
\mbox{$B \backslash  ({\rm range} \,Z^D_1 \cup \ldots \cup Z^D_{\frac{\alpha}{(1 + \wh{\delta})^2} \, {\rm cap}(D)}$) and}
\\[1ex]
\mbox{$B' \backslash  ({\rm range} \,Z^{D'}_1 \cup \ldots \cup Z^{D'}_{\frac{\alpha}{(1 + \wh{\delta})^2}  \, {\rm cap}(D')}$) with}
\\[1ex]
\mbox{$B'$ some neighboring box of $B$, which are not connected}
\\[0.5ex]
\mbox{by a path in $D \backslash ({\rm range} \,Z_1^D \cup \ldots \cup {\rm range} \,Z^D_{(1 + \wh{\delta})^2 \beta\, {\rm cap}(D)})\big]$},
\end{array}
\end{equation}
and likewise, by (\ref{4.12}) ii) (with $\ell = [\frac{\beta}{(1 + \wh{\delta})} \,{\rm cap}(D)]$),
\begin{equation}\label{5.19}
\begin{array}{l}
\IQ^{\cC_y} [\wh{c}_B] \le \IQ^{\cC_y}[\wt{G}^c_B] +\IQ^{\cC_y} [\wh{c}_B  \cap \wt{G}_B] \stackrel{(\ref{4.12})\, {\rm ii)}}{\le} \IQ^{\cC_y} [\wh{G}^c_B] \; +
\\[1ex]
\IP \Big[\dsl_{1 \le \ell \le \frac{\beta}{( 1 + \wh{\delta})^2}\, {\rm cap}(D)} \dil^{T_U}_0 e_D\big(Z^D_\ell (s)\big) ds < \gamma \,{\rm cap}(D)\Big].
\end{array}
\end{equation}

\n
Collecting (\ref{5.17}) - (\ref{5.19}), we see that for large $N$, for any $y$ in $\{0,L,2L,\dots, (\ov{K} - 1)L)^d\}$ and $B$ in $\cC_y$ we have
\begin{equation}\label{5.20}
\begin{array}{l}
\IQ^{\cC_y} [\wh{a}_B \cup \wh{b}_B \cup \wh{c}_B] \le 3 \IQ^{\cC_y} [\wt{G}^c_B] + \IP\big[\mbox{$B$ is bad$\big((1 + \wh{\delta})^2 \alpha,\beta,\gamma\big)\big]$} \; +
\\[1ex]
\mbox{$\IP\Big[B$ is bad$\Big(\mbox{\f $\dis\frac{\alpha}{(1 + \wh{\delta})^2}$}, (1 + \wh{\delta})^2 \beta, \gamma\Big)\Big] +  \IP\Big[B$ is bad$\Big(\alpha, \mbox{\f $\dis\frac{\beta}{(1 + \wh{\delta})^2}$}, \gamma\Big)\Big] \stackrel{(\ref{5.4})}{=} \eta$}
\end{array}
\end{equation}

\medskip\n
(actually a more careful bound of the left-hand side of (\ref{5.20}) yields an inequality without the factor $3$ in the first line of (\ref{5.20}), but this is irrelevant for our purpose).

\medskip
Keeping in mind the independence stated below (\ref{5.16}), we thus find with a similar notation as in (\ref{5.10}) that
\begin{equation}\label{5.21}
A_2 \le \ov{K}^d \,\IP\Big[\dsl^m_1 Y_i \ge \mbox{\f $\dis\frac{\rho}{2 \ov{K}^d}$} \;\Big(\mbox{\f $\dis\frac{N}{L}$}\Big)^{d-1}\Big].
\end{equation}

\n
As already mentioned below (\ref{5.10}), the quantity on the right-hand side of (\ref{5.21}) has super-exponential decay at rate $N^{d-2}$, so that
\begin{equation}\label{5.22}
\lim\limits_N \; \mbox{\f $\dis\frac{1}{N^{d-2}}$} \;\log A_2 = - \infty .
\end{equation}
Combining (\ref{5.9}), (\ref{5.12}), (\ref{5.22}), we have thus completed the proof of Theorem \ref{theo5.1}.
\end{proof}

\section{Disconnection upper bounds}
\setcounter{equation}{0}

In this section, we derive in Theorem \ref{theo6.3} the main asymptotic upper bound (\ref{0.6}) on the probability that random interlacements at level $u \in (0, \ov{u})$ disconnect $B_N$ from $S_N$. The proof involves the super-exponential estimate of the previous section, the occupation-time bounds of Section 3, and a coarse-graining procedure in the spirit of the proof of \hbox{Theorem 5.5} of \cite{Szni}. Once Theorem \ref{theo6.3} is proved, we quickly obtain in Corollary \ref{cor6.4} the main upper-bound (\ref{0.8}) on the probability that simple random walk disconnects $B_N$ from $S_N$. The argument is similar to the proof of Corollary 7.3 of \cite{Szni}, and involves letting $u$ tend to $0$ in Theorem \ref{theo6.3} and a natural coupling of $\cV^u$ under $\IP[\cdot \, |\,0 \in \cI^u]$ with the complement $\cV$ of the trace on $\IZ^d$ of the simple random walk under $P_0$.

\medskip
We begin with a simple connectivity lemma that will be helpful in the proof of the main Theorem \ref{theo6.3}. We use the notation of Section 2, and for the time being $L \ge 1$, $K \ge 100$ are integers, see (\ref{2.7}), boxes are defined as in (\ref{2.9}), (\ref{2.10}), the definition of good$(\alpha, \beta, \gamma)$ for an $L$-box $B$ appears in (\ref{2.11}) - (\ref{2.13}), and $N_u(D)$ is defined in (\ref{2.14}).

\begin{lemma}\label{lem6.1} $(\alpha > \beta > \gamma > 0$ and $u > 0)$

\medskip
If $B^i$, $0 \le i \le n$, is a sequence of neighboring $L$-boxes, which are good$(\alpha, \beta, \gamma)$, and $N_u(D^i) < \beta \, {\rm cap}(D)$, for $i = 0,\dots,n$ (with $D^i$ the $D$-type box attached to $B^i$), then
\begin{equation}\label{6.1}
\begin{array}{l}
\mbox{there exists a path in $\textstyle \bigcup\limits^n_{i=0}\, (D^i \backslash ({\rm range}\,Z^{D^i}_1 \cup \ldots \cup {\rm range}\, Z^{D^i}_{\beta\, {\rm cap}(D^i)}) \subseteq$}
\\[1ex]
\mbox{$\Big( \textstyle \bigcup\limits^n_{i=0} D^i\Big) \cap \cV^u$ starting in $B^0$ and ending in $B^n$}
\end{array}
\end{equation}
(note that ${\rm cap}(D^i) = {\rm cap}(D_0)$ for all $0 \le i \le n$, in the notation of (\ref{2.9})).
\end{lemma}

\begin{proof}
Since each $B^i$ is good$(\alpha, \beta, \gamma)$, $B^i \backslash ({\rm range}\,Z^{D^i}_1 \cup \ldots \cup {\rm range}\, Z^{D^i}_{\alpha\, {\rm cap}(D^i)})$ contains a connected set with diameter at least $\frac{L}{10}$. By (\ref{2.12}) the corresponding connected sets for $i$ and $i+1$ (where $0 \le i < n$) can be linked by a path in $D^i \backslash ({\rm range}\,Z^{D^i}_1 \cup \ldots \cup {\rm range}\, Z^{D^i}_{\beta\, {\rm cap}(D^i)})$). As a result we have a path from $B^0$ to $B^n$ in $\bigcup^{n-1}_{i=0} D^i \backslash ({\rm range}\,Z^{D^i}_1 \cup \ldots \cup\, {\rm range}\, Z^{D^i}_{\beta\, {\rm cap}(D)})$, and since ${\rm cap}(D^i) = {\rm cap}(D)$ and $N_u(D^i) < \beta \,{\rm cap}(D)$, for $i = 0,\ldots,n$, this last set is contained in $\bigcup^{n-1}_{i=0} (D^i \cap \cV^u) = (\bigcup^{n-1}_{i=0} D^i) \cap \cV^u$, and the lemma is proved.
\end{proof}
           
\begin{remark}\label{rem6.2} \rm The above connectivity result plays the role of (\ref{5.14}) of Lemma 5.9 of \cite{Szni}. The notion of an $L$-box $B$ being $\psi$-good at level $\alpha, \beta$ in the context of \cite{Szni} is replaced here by $B$ is good$(\alpha, \beta, \gamma)$, whereas the notion of an $L$-box being $h$-good at level $a$ in \cite{Szni}, is replaced here by the condition $N_u(D) < \beta\, {\rm cap}(D)$. \hfill $\square$
\end{remark}   

\medskip
We recall that $M > 1$ is a real number, $S_N = \{x \in \IZ^d, |x|_\infty = [MN]\}$, and $\ov{u}$ is defined in (\ref{2.3}). We now come to the main result.

\begin{theorem}\label{theo6.3} Assume that $0 < u < \ov{u}$. Then, for all $M > 1$,
\begin{equation}\label{6.2}
\limsup\limits_{N} \;\mbox{\f $\dis\frac{1}{N^{d-2}}$} \; \log \IP[A_N] \le - \mbox{\f $\dis\frac{1}{d}$} \;(\sqrt{\ov{u}} - \sqrt{u})^2 {\rm cap}_{\IR^d}([-1,1]^d),  
\end{equation}
where $A_N = \{B_N \overset{\cV^u}{\nleftrightarrow} S_N\}$ and ${\rm cap}_{\IR^d}([-1,1]^d)$ stands for the Brownian capacity of $[-1,1]^d$.
\end{theorem}

\begin{proof}
We first assume $M \ge 2$, pick $a \in (0,\frac{1}{10})$, and introduce the event
\begin{equation}\label{6.3}
\wh{A}_N = \{B_{(1 + a)N} \; \overset{^{\mbox{\footnotesize $\cV^u$}}}{\mbox{\Large $\longleftrightarrow$}} \hspace{-4ex}/ \quad \;\;S_N\}.
\end{equation}

\n
As a main step to (\ref{6.2}), we will first show that for all $M \ge 2$ and $0 < a < \frac{1}{10}$,
\begin{equation}\label{6.4}
\limsup\limits_{N} \;\mbox{\f $\dis\frac{1}{N^{d-2}}$} \; \log \IP[\wh{A}_N] \le - \mbox{\f $\dis\frac{1}{d}$} \;(\sqrt{\ov{u}} - \sqrt{u})^2 {\rm cap}_{\IR^d}([-1,1]^d).
\end{equation}
The claim (\ref{6.2}) will then quickly follow.

\medskip
We now prove (\ref{6.4}). We pick $\Gamma \ge 1$ (large), $\alpha > \beta > \gamma$ in ($u,\ov{u})$, as well as $\ve$ in $(0,1)$ such that $\ve (\sqrt{\frac{\ov{u}}{u}} - 1) < 1$ and $u(1 - \ve (\sqrt{\frac{\ov{u}}{u}} - 1))^{-2} < \gamma$. We then pick $L = [(\Gamma N \log N)^{\frac{1}{d-1}}]$ as in (\ref{5.1}) and $K \ge c_4(\alpha, \beta, \gamma) \vee c_5(\ve) \vee c_8(\alpha,\beta,\gamma)$, in the notation of Theorems \ref{theo2.3},  \ref{theo3.2} and \ref{theo5.1}. We then define $\eta$ as in (\ref{5.4}) and $\rho$ as in (\ref{5.6}).

\medskip
We recall the notation from below (\ref{5.1}) concerning columns. We know from Lemma \ref{lem6.1} that for large $N$, if all boxes in a column are good$(\alpha, \beta, \gamma)$, and the variables $N_u(D)$ corresponding to the various boxes in the column are smaller than $\beta \,{\rm cap}(D)$, then there exists a path in $\cV^u$ from $B_{(1 + a)N}$ to $S_N$ (here, we tacitly use of our choice of $a \in (0,\frac{1}{10})$ and $M \ge 2 > 1+a$). Hence, for large $N$, on $\wh{A}_N$ all columns contain an $L$-box, which is bad$(\alpha, \beta, \gamma)$, or such that $N_u(D) \ge \beta \,{\rm cap}(D)$.

\medskip
By definition of $C_N$ in (\ref{5.7}), we see that for large $N$,
\begin{equation}\label{6.5}
\begin{array}{l}
\mbox{on $D_N = \wh{A}_N \backslash C_N$, except for at most $\rho (\frac{N}{L})^{d-1}$ columns,}
\\
\mbox{all columns contain a box such that $N_u(D) \ge \beta \,{\rm cap}(D)$}.
\end{array}
\end{equation}

\n
Thus, on the event $D_N$ we can throw away $[\rho (\frac{N}{L})^{d-1}]$ columns, and then, for each column in the remaining set of columns, select a box $B$, which is good$(\alpha, \beta, \gamma)$, and with corresponding $N_u(D) \ge \beta\,{\rm cap}(D)$. We further remove columns that have their projection on the face $F_{e,N}$ attached to the column, at distance less than $\ov{K} L$ from any other face $F_{e',N}$, $e' \not= e$ (recall that $\ov{K} = 2 K + 3$, see (\ref{3.1})). Then, restricting to a sub-lattice, we only keep columns attached to a given face $F_{e,N}$, with $|e| = 1$, which are at mutual $|\cdot|_\infty$-distance at least $\ov{K}L$ (specifically, we only keep columns of boxes that have labels $z \in \IL$ such that the projection of $z$ on the orthogonal space to $e$ belongs to $\ov{K} \IL$).

\medskip
We write $\wt{C}$ for the subset of $\partial_i B_N = \bigcup_{|e| = 1} F_{e,N}$ (the internal boundary of $B_N$) obtained by projecting the selected columns onto the face $F_{e,N}$ attached to the respective columns, and write $F_N$ for the set of points of $\partial_i B_N$ that belong to a single face $F_{e,N}$ and are at $|\cdot|_\infty$-distance at least $\ov{K} L$ from all other faces $F_{e',N}$, $e' \not= e$.

\medskip
In the fashion described above, selecting a subset of columns in at most $2^{c(\frac{N}{L})^{d-1}}$ ways and in each selected column a box in at most $(M+1) N/L$ ways, we see that there is a family with cardinality at most $\exp\{c_9(\frac{N}{L})^{d-1} \log ((M+1)N)]$ of finite subsets $\cC$ of $L$-boxes such that
\begin{equation}\label{6.6}
\left\{ \begin{array}{rl}
{\rm i)} & \mbox{the boxes $B$ in $\cC$ belong to mutually distinct columns},
\\[1ex]
{\rm ii)} & \mbox{the columns containing a box of $\cC$ are at mutual $|\cdot |_\infty$-distance}\\
& \mbox{at least $\ov{K}L$},
\\[1ex]
{\rm iii)} & \wt{C} \subseteq F_N,
\\[1ex]
{\rm iv)} & \mbox{at most $c_{10}(\ov{K}) (N^{d-2}L + \rho N^{d-1})$ points of $F_N$ are at $|\cdot |_\infty$-distance}\\
&\mbox{bigger than $\ov{K}L$ of $\wt{C}$},
\end{array}\right.
\end{equation}
\n
(recall that $\wt{C}$ is obtained by projecting the boxes of $\cC$ on the face of $B_N$ attached to the column where the box sits).

\medskip
The above procedure yields a ``coarse graining'' of the event $D_N$, in the sense that for large $N$,
\begin{equation}\label{6.7}
D_N \subseteq \textstyle \bigcup\limits_{\cC} D_{N,\cC}, \;\mbox{where $D_{N,\cC} = \bigcap\limits_{B \in \cC} \{B$ is good$(\alpha,\beta,\gamma)$ and $N_u(D) \ge \beta \,{\rm cap}(D)\}$}
\end{equation}
(and $\cC$ runs over a family with cardinality at most $\exp\{c_9 (\frac{N}{L})^{d-1} \log ((M+1)N)\}$ with (\ref{6.6}) fulfilled).

\medskip
By Theorem \ref{theo3.2}, we can bound $\IP[D_{N,\cC}]$ uniformly in the above collection and find that
\begin{equation}\label{6.8}
\begin{array}{l}
\limsup\limits_N \; \mbox{\f $\dis\frac{1}{N^{d-2}}$} \;\log \IP[D_N] \le 
\\[1ex]
\limsup\limits_N \; \bigg\{ c_9 \Big(\mbox{\f $\dis\frac{N}{L}$}\Big)^{d-1} \;\mbox{\f $\dis\frac{\log((M+1)N)}{N^{d-2}}$} \; -  
\\[2ex]
\Big(\sqrt{\gamma} - \frac{\sqrt{u}}{1- \ve(\sqrt{\frac{\ov{u}}{u}} - 1)} \Big) \; (\sqrt{\gamma} - \sqrt{u}) \; \inf\limits_{\cC} \; \mbox{\f $\dis\frac{{\rm cap}(\ov{C})}{N^{d-2}}$}\,\bigg\} \stackrel{(\ref{5.1})}{\le} \mbox{\f $\dis\frac{c_9}{\Gamma}$} \; -
\\[2ex]
(\sqrt{\gamma} - \sqrt{\wt{u}})^2 \; \liminf\limits_N \; \mbox{\f $\dis\frac{1}{N^{d-2}}$} \; \inf\limits_{\cC} \; {\rm cap}(\ov{C}),
\end{array}
\end{equation}

\n
with $\wt{u} = u(1 - \ve( \mbox{\f $\sqrt{\mbox{\normalsize $\frac{\ov{u}}{u}$}}$} - 1))^{-2}$ $< \gamma$ (by our choice of $\ve$ below (\ref{6.4})), and where we have set $\ov{C} = \bigcup_{B \in \cC} B$ and used that $\ov{C} \subseteq C = \bigcup_{B \in \cC} D$ in the notation of (\ref{3.3}) (the upper bound in (\ref{3.15}) of Theorem \ref{theo3.2} is actually stronger than what we use in (\ref{6.8})).

\medskip
By the same arguments leading to (5.37) of \cite{Szni} and Lemma 5.6 of the same reference (which is based on a Wiener-type criterion), we know that
\begin{equation}\label{6.9}
\liminf\limits_N \; \inf\limits_\cC \; \mbox{\f $\dis\frac{{\rm cap}(\ov{C})}{{\rm cap}(\wt{C})}$} \ge 1 \;\mbox{and} \; \liminf\limits_N \; \inf\limits_{\cC} \;  \mbox{\f $\dis\frac{{\rm cap}(\wt{C})}{{\rm cap}(B_N)}$}  \ge 1.
\end{equation}
(incidentally, the last inequality actually is an equality).

As a result, coming back to (\ref{6.8}), we obtain that
\begin{equation}\label{6.10}
\begin{split}
\limsup\limits_N \; \mbox{\f $\dis\frac{1}{N^{d-2}}$} \;\log \IP[D_N]  & \le \mbox{\f $\dis\frac{c_9}{\Gamma}$} - (\sqrt{\gamma} - \sqrt{\wt{u}})^2 \; \liminf\limits_N \; \mbox{\f $\dis\frac{{\rm cap}(B_N)}{N^{d-2}}$}
\\
& =  \mbox{\f $\dis\frac{c_9}{\Gamma}$} - \mbox{\f $\dis\frac{1}{d}$}\; (\sqrt{\gamma} - \sqrt{\wt{u}})^2 \;{\rm cap}_{\IR^d} ([ -1,1]^d),
\end{split}
\end{equation}

\medskip\n
where we used (3.14) of \cite{Szni}, and recall that $\wt{u} = u(1 - \ve (\sqrt{\frac{\ov{u}}{u}} - 1))^{-2} ( < \gamma)$.

\medskip
We can now bring into play the super-exponential estimate of Theorem \ref{theo5.1} and find  \begin{equation*}
\begin{split}
\limsup\limits_N \; \mbox{\f $\dis\frac{1}{N^{d-2}}$} \;\log \IP[\wh{A}_N]  & \le \limsup\limits_N\;  \mbox{\f $\dis\frac{1}{N^{d-2}}$} \;\max (\log \IP[D_N], \log \IP[C_N])
\\
& \le  \mbox{\f $\dis\frac{c_9}{\Gamma}$} - \mbox{\f $\dis\frac{1}{d}$}\; (\sqrt{\gamma} - \sqrt{\wt{u}})^2 \;{\rm cap}_{\IR^d} ([ -1,1]^d).
\end{split}
\end{equation*}

\n
Letting $\Gamma$ tend to infinity, $\ve$ tend to zero, and then $\gamma$ tend to $\ov{u}$, we find (\ref{6.4}).

\medskip
We will now prove (\ref{6.2}). We introduce $N' = [\frac{N}{1 + a}]$ with $a \in (0,\frac{1}{10})$ as above (\ref{6.4}), and $M' = 2(1+ a)M$ (so that $M' \ge 2$). We denote by $A'_N$ the event where $N$ in (\ref{6.3}) is replaced by $N'$ and $M$ (entering the definition of $S_N$) is replaced by $M'$. For large $N$, we have $(1 + a)N' \le N$ and $M' N' \ge MN$, so that $A_N \subseteq A_{N'}$. As a result of this inclusion and (\ref{6.4}) we see that
\begin{equation}\label{6.11}
\begin{split}
\limsup\limits_N \; \mbox{\f $\dis\frac{1}{N^{d-2}}$} \;\log \IP[A_N]  & \le \limsup\limits_N\;  \Big(\mbox{\f $\dis\frac{N'}{N}$} \Big)^{d-2}\;   \mbox{\f $\dis\frac{1}{N'^{(d-2)}}$} \;\log \IP[A'_N] 
\\
&\!\! \stackrel{(\ref{6.4})}{\le} -  \mbox{\f $\dis\frac{1}{(1 + a)^{d-2}}$} \;\mbox{\f $\dis\frac{1}{d}$}\; (\sqrt{\ov{u}} - \sqrt{{u}})^2 \;{\rm cap}_{\IR^d} ([ -1,1]^d).
\end{split}
\end{equation}

\n
Letting $a$ tend to zero, we obtain (\ref{6.2}). This concludes the proof of Theorem \ref{theo6.3}.
\end{proof}
        
We will now deduce from Theorem \ref{theo6.3} an asymptotic upper bound on the probability that simple random walk disconnects $B_N$ from $S_N$. We denote by $\cI$ the set of points visited by the simple random walk and by $\cV = \IZ^d \backslash \cI$ its complement. We recall that $P_0$ governs the law of the walk starting from the origin, and $S_N = \{x \in \IZ^d; |x|_\infty = [MN]\}$. 

\begin{corollary}\label{cor6.4} $(d \ge 3$, $\ov{u}$ as in (\ref{2.3}))

\medskip
For any $M > 1$,
\begin{equation}\label{6.12}
\limsup\limits_N \; \mbox{\f $\dis\frac{1}{N^{d-2}}$} \;\log \IP_0 [B_N \;\overset{^{\mbox{\footnotesize $\cV$}}}{\mbox{\Large $\longleftrightarrow$}} \hspace{-4ex}/ \quad \;\;S_N] \le \mbox{\f $-\dis\frac{1}{d}$} \; \ov{u} \,{\rm cap}_{\IR^d}([-1,1]^d).
\end{equation}
\end{corollary}
           
\begin{proof}
The argument is similar to that of Corollary 7.3 of \cite{Szni}. Given $u > 0$, we can find a coupling $P$ of $\cI^u$ under $\IP[\cdot \,| 0 \in \cI^u]$ and $\cI$ under $P_0$, so that $P$-a.s., $\cI \subseteq \cI^u$. As a result, we have
\begin{equation}\label{6.13}
\begin{split}
P_0 [B_N \;\overset{^{\mbox{\footnotesize $\cV$}}}{\mbox{\Large $\longleftrightarrow$}} \hspace{-4ex}/ \quad \;\;S_N] & = P [B_N \;\overset{^{\mbox{\footnotesize $\cV$}}}{\mbox{\Large $\longleftrightarrow$}} \hspace{-4ex}/ \quad \;\;S_N] \le P [B_N \;\overset{^{\mbox{\footnotesize $\cV^u$}}}{\mbox{\Large $\longleftrightarrow$}} \hspace{-4ex}/ \quad \;\;S_N]
\\
& = \IP [B_N\;\overset{^{\mbox{\footnotesize $\cV^u$}}}{\mbox{\Large $\longleftrightarrow$}} \hspace{-4ex}/ \quad \;\;S_N\,| 0 \in \cI^u] \le (1 - e^{-\frac{u}{g(0)}})^{-1}\, \IP [B_N\;\overset{^{\mbox{\footnotesize $\cV^u$}}}{\mbox{\Large $\longleftrightarrow$}} \hspace{-4ex}/ \quad \;\; S_N],
\end{split}
\end{equation}

\medskip\n
with $g(0)$ as in (\ref{1.4}) (and $e^{-\frac{u}{g(0)}} = \IP[0 \in \cV^u]$). By Theorem \ref{theo6.3}, it now follows that for any $0 < u < \ov{u}$,
\begin{equation}\label{6.14}
\limsup\limits_N \; \mbox{\f $\dis\frac{1}{N^{d-2}}$} \;\log P_0 [B_N \;\overset{^{\mbox{\footnotesize $\cV$}}}{\mbox{\Large $\longleftrightarrow$}} \hspace{-4ex}/ \quad \;\;S_N] \le \mbox{\f $-\dis\frac{1}{d}$} \;(\sqrt{\ov{u}} - \sqrt{u})^2 \,{\rm cap}_{\IR^d}([-1,1]^d).
\end{equation}
Letting $u \rightarrow 0$ yields (\ref{6.12}).
\end{proof}

\begin{remark}\label{rem6.5} \rm ~  

\medskip\n
1) As mentioned in the Introduction, it is plausible, but not known at the moment, that $\ov{u} = u_* = u_{**}$. If this is the case, then \cite{LiSzni14} and \cite{Li} yield matching asymptotic lower bounds for Theorem \ref{theo6.3} and for Corollary \ref{cor6.4}.

\medskip\n
2) When $K$ is a (suitably regular) compact subset of $\IR^d$, and $K_N = (NK) \cap \IZ^d$ its discrete blow-up, one can wonder whether in the case of random interlacements on $\IZ^d$ (with hopefully obvious notation)
\begin{equation}\label{6.15}
\lim\limits_N \;\mbox{\f $\dis\frac{1}{N^{d-2}}$} \;\log \IP[K_N \;\overset{^{\mbox{\footnotesize $\cV^u$}}}{\mbox{\Large $\longleftrightarrow$}} \hspace{-4ex}/ \quad \;\;\infty] = -\mbox{\f $\dis\frac{1}{d}$} \;(\sqrt{u}_* - \sqrt{u})^2 \,{\rm cap}_{\IR^d}(K), \;\mbox{for $0 < u < u_*$},
\end{equation}
and in the case of simple random walk on $\IZ^d$,
\begin{equation}\label{6.16}
\lim\limits_N \;\mbox{\f $\dis\frac{1}{N^{d-2}}$} \;\log \IP_0[K_N \;\overset{^{\mbox{\footnotesize $\cV$}}}{\mbox{\Large $\longleftrightarrow$}} \hspace{-4ex}/ \quad \;\;\infty] = -\mbox{\f $\dis\frac{1}{d}$} \;u_*\,{\rm cap}_{\IR^d}(K).
\end{equation}

\n
Asymptotic lower bounds, with liminf in place of lim, and $u_{**}$ in place of $u_{*}$, are shown in \cite{LiSzni14}, in the case of (\ref{6.15}), and in \cite{Li}, in the case of (\ref{6.16}).

\bigskip\n
3)
Possibly, some of the techniques developed in this article might be helpful to improve the results of \cite{Wind08b} concerning the disconnection time of simple random walk with a slight bias in a discrete cylinder with a large periodic base (see Theorems 1.1, 1.2 and Remark 6.7 of \cite{Wind08b}). \hfill $\square$

\end{remark}

\appendix
\section{Appendix}
\setcounter{equation}{0}

In this appendix we provide the proof of Lemma \ref{lem1.3}. The arguments are similar to p.~50 of 
\cite{Lawl91}. The proof is included for the reader's convenience.

\bigskip\n
{\it Proof of Lemma \ref{lem1.3}.} The uniqueness in the statement is immediate. Only the existence is at stake. We first observe that if $z \in \partial U$ does not belong to the boundary the connected component of $U$ containing $B(0,L)$, all three members of (\ref{1.21}) vanish (and (\ref{1.21}) holds with $\psi^{A,U}_{y,z} = 0$). We can thus assume from now on that $U$ is connected.

\medskip
The first equality $P_z[H_A < \wt{H}_{\partial U}, X_{H_A} = y] = P_y[T_U < \wt{H}_A, X_{T_U} = z]$, with $y \in A$ and $z \in \partial U$, is obtained by time-reversal (one sums over the possible values in the discrete skeleton of the walk, of the entrance time in $A$, for the probability on the left, and of the exit time of $U$, for the probability on the right). It then suffices to show that when $K \ge c$, for $y \in A$, $z \in \partial U$, one has
\begin{equation}\label{A.1}
P_y[T_U < \wt{H}_A, X_{T_U} = z] = e_A(y) \,P_0 [X_{T_U} = z] (1 + \psi), \;\mbox{with $|\psi | \le c'/K$}.
\end{equation}
One introduces the function on $\IZ^d$
\begin{equation*}
h(x) = P_x[X_{T_U} = z], \;\;\mbox{for $x \in \IZ^d$}.
\end{equation*}

\n
It is positive and harmonic in $U \supseteq B(0,KL)$. When $K \ge c$, it follows from the gradient control in Theorem 1.7.1, on p.~42 of \cite{Lawl91}, and the Harnack inequality in Theorem 1.7.2, on p.~42 of \cite{Lawl91}, that for $x  \in D \cup \partial D$, where $D = B(0,2L)$,
\begin{equation}\label{A.2}
|h(x) - h(0)| \le c \mbox{\f $\dis\frac{|x|}{KL}$} \; \sup\limits_{B(0,\frac{KL}{2})} h \le c \mbox{\f $\dis\frac{|x|}{KL}$} \;h(0)
\end{equation}
(using chaining in the last step).

\medskip
Noting that $T_D$ happens before $T_U$, we find that by the strong Markov property
\begin{equation}\label{A.3}
P_y[T_U  < \wt{H}_A, X_{T_U} = z] = E_y\big[T_D < \wt{H}_A, P_{X_{T_D}} [T_U < H_A, X_{T_U} = z]\big].
\end{equation}
Now, for $x \in \partial D$ we have
\begin{equation}\label{A.4}
\begin{array}{l}
P_x [T_U < H_A, X_{T_U} = z] = h(x) - P_x[H_A < {T_U}, X_{T_U} = z] =
\\
h(x) - P_x[H_A + T_D \circ \theta_{H_A} < T_U, X_{T_U} = z] \stackrel{\rm strong\;Markov}{=} 
\\[0.5ex]
h(x) - E_x[H_A + T_D \circ \theta_{H_A} < T_U, h(X_{T_D} \circ \theta_{H_A})] =
\\[0.5ex]
h(x) - E_x[H_A < T_U, h(X_{T_D} \circ \theta_{H_A})].
\end{array}
\end{equation}
As a result of (\ref{A.2}) we see that
\begin{equation}\label{A.5}
\varphi = \mbox{\f $\dis\frac{1}{h(0)}$} \;\big(h - h(0)\big) \; \mbox{satisfies} \; \sup\limits_{D \cup \partial D} |\varphi| \le \wt{c}/K.
\end{equation}
In addition, it follows by (\ref{A.4}) that for $x \in \partial D$
\begin{equation}\label{A.6}
P_x[T_U < H_A, X_{T_U}= z] = h(0) (P_x[T_U < H_A] + \varphi(x) - E_x[H_A < T_U, \varphi(X_{T_D} \circ \theta_{H_A})]).
\end{equation}
Further, we note that for $x \in \partial D$
\begin{equation}\label{A.7}
P_x[T_U < H_A] \ge P_x[H_{B(0,L)} = \infty] \ge c,
\end{equation}
and setting for $x \in \partial D$
\begin{equation}\label{A.8}
\wt{
\varphi}(x) = \varphi(x) / P_x[T_U < H_A], \; \mbox{we have} \; \sup\limits_{\partial D} |\wt{\varphi}| \le c/K.
\end{equation}
As a result, we obtain that for $x \in \partial D$
\begin{equation}\label{A.9}
P_x[T_U < H_A, X_{T_U}= z] = h(0) \, P_x[T_U < H_A] \big(1 + \gamma(x)\big),
\end{equation}
where we have set
\begin{equation}\label{A.10}
\gamma(x) = \wt{\varphi}(x) - \mbox{\f $\dis\frac{1}{P_x[T_U < H_A]}$}  \; E_x[H_A < T_U, \varphi(X_{T_D} \circ \theta_{H_A})], \; \mbox{so that $\sup\limits_{\partial D} |\gamma | \le c/K$}.
\end{equation}     

\n
Coming back to (\ref{A.3}) we obtain (with $X_{T_D}$ playing the role of $x$ in (\ref{A.9})) that
\begin{equation}\label{A.11}
\begin{array}{l}
P_y[T_U < \wt{H}_A, X_{T_U} = z] = h(0) \,E_y\big[T_D < \wt{H}_A, P_{X_{T_D}} [T_U < H_A] \big(1 + \gamma(X_{T_D})\big)\big] =
\\[1ex]
P_0[X_{T_U} = z] \,e_A(y) \;\mbox{\f $\dis\frac{P_y[T_U < \wt{H}_A]}{P_y[\wt{H}_A = \infty]}$} \; E_y [1 + \gamma(X_{T_D})\,| \,T_U < \wt{H}_A],
\end{array}
\end{equation}

\n
where we note that when $e_{A,U}(y) = P_y [T_U < \wt{H}_A] > 0$, then $e_A(y) = P_y[\wt{H}_A = \infty] > 0$ (because $e_{A,U}(y) \le e_{A,B(0,KL)}(y)$ and (\ref{1.16})), and the ratio above is understood as $1$ if $e_A(y)$ vanishes. We thus see that the ratio in the last line of (\ref{A.11}) lies between $1$ and $1 + \rho_{A,B(0,KL)} \le (1 - \frac{c}{K^{d-2}})^{-1}_+$, by (\ref{1.18}) of Lemma \ref{lem1.2} and (\ref{1.8}). 

\medskip
Thus, looking at the last line of (\ref{A.11}), we see that when $K \ge c$, 
\begin{equation*}
P_y[T_U < \wt{H}_A, X_{T_U} = z] = e_A(y) \,P_0[X_{T_U} = z] (1 + \psi) \; \mbox{with} \; |\psi | \le c'/K,
\end{equation*}
i.e. (\ref{A.1}) holds and this proves Lemma \ref{lem1.3}. \hfill $\square$

\end{document}